\documentclass{article}
\usepackage{guff}
\usepackage{pictures}

\renewcommand{\leq}{\leqslant}
\renewcommand{\geq}{\geqslant}
\renewcommand{\subset}{\subseteq}

\title{\vspace{-0.7cm}Trees with few leaves in tournaments}
\author{Alistair\ Benford\thanks{School of Mathematics, University of Birmingham,
Birmingham,
B15 2TT,
UK.
AXB1433@bham.ac.uk}
\ and\ Richard\ Montgomery\thanks{School of Mathematics,
University of Birmingham,
Birmingham,
B15 2TT,
UK. r.h.montgomery@bham.ac.uk. Research supported by the European Research Council (ERC) under the European Union's Horizon
2020 research and innovation programme (grant agreement no.\ 947978), and the Leverhulme Trust.}}
\date{}

\begin{document}

\maketitle

\begin{abstract}
We prove that there exists $C>0$ such that any $(n+Ck)$-vertex tournament contains a copy of every $n$-vertex oriented tree with $k$ leaves, improving the previously best known bound of $n+O(k^2)$ vertices to give a result tight up to the value of $C$. Furthermore, we show that, for each $k$, there exists $n_0$, such that, whenever $n\geqslant n_0$, any $(n+k-2)$-vertex tournament contains a copy of every $n$-vertex oriented tree with at most $k$ leaves, confirming a conjecture of Dross and Havet.
\end{abstract}

\section{Introduction}
The study of trees in tournaments has been motivated largely by Sumner's universal tournament conjecture from 1971, which states that every $(2n-2)$-vertex tournament should contain a copy of every $n$-vertex oriented tree (see, e.g.,~\cite{reid1983embedding}). In 1991, H\"aggkvist and Thomason~\cite{HAE-THO} gave the first proof that $O(n)$ vertices in a tournament are sufficient to find any $n$-vertex oriented tree. Following several subsequent improvements to the implicit constant~\cite{havet2000median,HAV,El_S}, Dross and Havet~\cite{DRO-HAV} recently showed that $\big\lceil\frac{21}{8}n-\frac{47}{16}\big\rceil$ vertices are in fact sufficient, giving the best known bound which holds for all $n$. On the other hand, Sumner's conjecture is known to be true for sufficiently large $n$, as shown in 2010 by K\"uhn, Mycroft and Osthus~\cite{KUE-MYC-OST-2}, using regularity methods.

If true, Sumner's conjecture would be tight for each $n$, as demonstrated by the $n$-vertex star with every edge oriented out from the root vertex. The appearance of many trees can, however, be ensured with far fewer than $2n-2$ vertices in the tournament. Indeed, confirming a conjecture of Rosenfeld~\cite{ROS}, Thomason~\cite{THO} showed in 1986 that there is some $n_0$ such that, whenever $n\geqslant n_0$, any $n$-vertex tournament contains a copy of every $n$-vertex oriented path. In 2000, Havet and Thomass\'e~\cite{HAV-THO} showed that the optimal value of $n_0$ is $8$, a result recently given a shorter proof by Hanna~\cite{HAN}.

Answering the natural question arising from the different behaviour here between stars and paths, H\"aggkvist and Thomason~\cite{HAE-THO} showed in 1991 that the number of additional vertices required in the tournament can be bounded as a function of the number of leaves in the tree. That is, for each $k$, there is some smallest $g(k)$ such that every $(n+g(k))$-vertex tournament contains a copy of every $n$-vertex tree with $k$ leaves. The upper bound shown by H\"aggkvist and Thomason on $g(k)$ is exponential in $k^3$, but was recently improved to $144k^2-280k+124$ by Dross and Havet~\cite{DRO-HAV}. Havet and Thomass\'e~\cite{HAV2} conjectured in 2000 that $g(k)\leq k-1$ for each $k\geq 2$. That is, generalising Sumner's conjecture, they conjectured that every $(n+k-1)$-vertex tournament contains a copy of every $n$-vertex oriented tree with $k$ leaves.

In this paper, we give the first linear bound on $g(k)$, as follows.

\begin{theorem}\label{thm:fewlinear}
There is some $C>0$ such that every $(n+Ck)$-vertex tournament contains a copy of every $n$-vertex oriented tree with $k$ leaves.
\end{theorem}

If true, Havet and Thomass\'e's conjecture would be tight whenever $k=n-1$ (i.e., whenever it is covered by Sumner's conjecture), but for general $n$ and $k$, we only have examples showing that the tournament may need to have at least $n+k-2$ vertices (as described below). From the result of Havet and Thomass\'e~\cite{HAV-THO} on oriented paths we know that $n+k-2$ is best possible if $k=2$ and $n\geq 8$, while Ceroi and Havet~\cite{CER-HAV} proved that $n+k-2$ is also best possible if $k=3$ and $n\geq 5$. Dross and Havet~\cite{DRO-HAV} conjectured that, for each $k$, if $n$ is sufficiently large then $n+k-2$ is best possible.

In this paper, we confirm this conjecture, as follows.

\begin{theorem}\label{thm:fewcst}
For each $k$, there is some $n_0$ such that, for each $n\geqslant n_0$, every $(n+k-2)$-vertex tournament contains a copy of every $n$-vertex oriented tree with $k$ leaves.
\end{theorem}

The following well-known example shows that this is best possible. Form a tree $T_{n,k}$ by taking a directed path $P$ with $n-k+1$ vertices and attaching $k-1$ out-leaves to the last vertex of $P$. The resulting oriented tree $T_{n,k}$ has $n$ vertices and $k$ leaves. Construct the following $(n+k-3)$-vertex tournament $G$. Let $V(G)=A\cup B$, where $A$ and $B$ are disjoint sets with $|A|=n-k$ and $|B|=2k-3$. Orient the edges of $G$ so that $G[B]$ is a regular tournament, $G[A]$ is an arbitrary tournament, and all edges are directed from $A$ to $B$. As $d_G^+(v)=k-2$ for each $v\in B$, if $G$ contains a copy of $T_{n,k}$ then the last vertex of $P$ must be copied to $A$. Then, as every edge between $A$ and $B$ is oriented into $B$, every vertex of $P$ must be copied into $A$, a contradiction as $|A|=n-k$. Thus the $n$-vertex tree $T_{n,k}$ with $k$ leaves does not appear in the $(n+k-3)$-vertex tournament $G$.

To prove Theorems~\ref{thm:fewlinear} and~\ref{thm:fewcst}, we use median orders, a technique first used to embed trees in tournaments by Havet and Thomass\'e~\cite{havet2000median}. In particular, we exploit the property that pairs of vertices in a median order can be robustly connected by directed paths with length 3 travelling in the direction of the order (see Lemma~\ref{lm:connecting_path_length_3}), using this repeatedly in our embeddings. We have not optimised the value of $C$ reachable with our methods as this will not reach a plausibly optimal bound, but we show that Theorem~\ref{thm:fewlinear} holds for some $C<500$. We do not calculate an explicit function $n_0(k)$ for Theorem~\ref{thm:fewcst}, but our methods show that we may take $n_0(k)=k^{O(k)}$. However, it seems likely some function $n_0(k)$ satisfying Theorem~\ref{thm:fewcst} with $n_0(k)=O(k)$ exists.

After stating our notation, in Section~\ref{sec:prelim}, we recall median orders and their basic properties, before proving Lemma~\ref{lm:connecting_path_length_3} and giving the other preliminary results we will use. In Section~\ref{sect:fewlinear}, after sketching its proof, we prove Theorem~\ref{thm:fewlinear}. In Section~\ref{sect:fewcst}, we give a proof sketch before proving Theorem~\ref{thm:fewcst}.

\section{Preliminaries}\label{sec:prelim}

\subsection{Notation}
For a directed graph (digraph) $G$, we use $V(G)$ to denote the vertex set of $G$ and $E(G)$ to denote the edge set of $G$. We write $|G|=|V(G)|$ for the order of $G$. Each element of $E(G)$ is an ordered pair $(u,v)$, where $u,v\in V(G)$, which we write as $uv$. If $uv\in E(G)$, then we say that \emph{$u$ dominates $v$} (written $u\rightarrow_G v$), that $v$ is an \emph{out-neighbour} of $u$, and that $u$ is an \emph{in-neighbour} of $v$. Given $v\in V(G)$, the \emph{out-neighbourhood} of $v$, written $N_G^+(v)$, is the set of out-neighbours of $v$ in $V(G)$, and the \emph{in-neighbourhood} of $v$, written $N_G^-(v)$ is the set of in-neighbours of $v$ in $V(G)$. Throughout, we use $+$ and $-$ interchangeably with `in' and `out' respectively. For each $\diamond\in\{+,-\}$, the \emph{$\diamond$-degree} of $v$ in $G$ is $d_G^\diamond(v)=|N_G^\diamond(v)|$. For $X,Y\subseteq V(G)$ and $\diamond\in\{+,-\}$, we write $N_G^\diamond(X)=(\cup_{v\in X}N_G^\diamond(v))\setminus X$ and $N_G^\diamond(X,Y)=N_G^\diamond(X)\cap Y$. We denote by $G[X]$ the induced sub-digraph of $G$ with vertex set $X$ and let $G-X=G[V(G)\setminus X]$. Subscripts are omitted wherever they are clear from context, as are rounding signs wherever they are not crucial.

An \emph{oriented graph} is a digraph with at most one edge between any pair of vertices. A \emph{tournament} $G$ is a digraph whose underlying graph is a complete graph, i.e., for each $u,v\in V(G)$ with $u\neq v$, exactly one of $uv$ or $vu$ is in $E(G)$. An \emph{oriented tree} (respectively, \emph{oriented path}) is a digraph whose underlying graph is a tree (respectively, path). A \emph{directed path} from $v_0$ to $v_\ell$ is a path of the form $v_0\rightarrow v_1\rightarrow \ldots\rightarrow v_\ell$. The \emph{length} of a path $P$ is $|P|-1$, and denoted $\ell(P)$.  We say a subpath $P$ of a forest $T$ is a \emph{bare path} if all of the internal vertices $v$ of $P$ have $d_T(v)=2$, and we denote by $T-P$ the digraph formed from $T$ by removing all the edges and internal vertices of $P$.

Having proved, for example, a result holds for $\diamond=+$, we will occasionally deduce the same result for $\diamond=-$ by \emph{directional duality}. That is, reversing all the relevant orientations and applying the result with $\diamond=+$ implies, after reversing the edges again, the result with $\diamond=-$.
We also use standard hierarchy notation. That is, for $a,b\in(0,1]$, we write $a\ll b$ to mean that there is a non-decreasing function $f:(0,1]\to(0,1]$ such that the subsequent statement holds whenever $a\leqslant f(b)$.

\subsection{Median orders}
Median orders were first used to embed trees in tournaments by Havet and Thomass\'e~\cite{havet2000median}. Given a tournament $G$, an ordering $\sigma =v_1,\ldots,v_n$ of $V(G)$ is a \emph{median order} if it maximises the number of pairs $i<j$ with $v_iv_j\in E(G)$. The following lemma gives two simple fundamental properties of median orders (see, e.g., \cite[Lemma 9]{DRO-HAV}).

\begin{lemma}\label{lm:basic_properties_of_median_orders}
Let $G$ be a tournament and $v_1,\ldots,v_n$ a median order of $G$. Then, for any two indices $i$, $j$ with $1\leqslant i<j\leqslant n$, the following properties hold.
\begin{enumerate}[label = \textbf{\roman{enumi})}]
    \item\label{median1} $v_i,v_{i+1},\ldots,v_j$ is a median order of the induced subtournament $G[\{v_i,v_{i+1},\ldots,v_j\}]$.
    \item\label{median2} $v_i$ dominates at least half of the vertices $v_{i+1},v_{i+2},\ldots,v_j$, and $v_j$ is dominated by at least half of the vertices $v_i,v_{i+1},\ldots,v_{j-1}$. In particular, each vertex $v_i$, $1\leqslant i<n$, dominates its successor $v_{i+1}$.
\end{enumerate}
\end{lemma}

Median orders contain short directed paths from any vertex to any vertex later in the order, as follows (in combination with Lemma~\ref{lm:basic_properties_of_median_orders} i)).

\begin{corollary}\label{cor:last} Let $n\geq 2$.
If $v_1,\ldots,v_n$ is a median order of the $n$-vertex tournament $G$, then $G$ contains a directed path from $v_1$ to $v_n$ with length at most 2.
\end{corollary}
\begin{proof} Suppose $v_1v_n\notin E(G)$, for otherwise such a path exists, and let $V=\{v_2,\ldots,v_{n-1}\}$. Then, by Lemma~\ref{lm:basic_properties_of_median_orders}~ii), $|N^+(v_1,V)|=|N^+(v_1)|\geq \frac{n-1}{2}>|V|/2$, and, similarly, $|N^-(v_n,V)|>|V|/2$. Therefore, there is some $w\in V$ such that $v_1wv_n$ is a directed path.
\end{proof}

Median orders have been used particularly effectively to embed arborescences in tournaments. An \emph{out-arborescence} (respectively, \emph{in-arborescence}) is an oriented tree $T$ with a root vertex $t\in V(T)$ such that, for every $v\in V(T)$, the path between $t$ and $v$ in $T$ is directed from $t$ to $v$ (respectively, from $v$ to $t$). Dross and Havet \cite{DRO-HAV} used median orders to prove that any $(n+k-1)$-vertex tournament contains a copy of any $n$-vertex arborescence with $k$ leaves. We will use their result in the following slightly stronger form (see~\cite[Theorem 12]{DRO-HAV}).
 
\begin{theorem}\label{thm:DH}
Let $A$ be an $n$-vertex out-arborescence with $k\geqslant 1$ out-leaves and root $r$. Let $G$ be a tournament on $n+k-1$ vertices and let $\sigma=v_1,\ldots,v_{n+k-1}$ be a median order of $G$. Then, there is an embedding $\phi$ of $A$ in $G$ such that $\phi(r)=v_1$.
\end{theorem}

We will need some linear bound on the number of vertices required in a tournament, which we then apply for small trees. Any linear bound would suffice, but with the value of $C$ in mind we derive Corollary~\ref{cor:DHbound} from the following theorem of Dross and Havet~\cite{DRO-HAV}.

\begin{theorem}\label{thm:DHbound} For each $n\geq 2$, every $\big\lceil \frac{3}{2}n+\frac{3}{2}k-2\big\rceil$-vertex tournament contains a copy of every $n$-vertex oriented tree with $k$ leaves.
\end{theorem}

\begin{corollary}\label{cor:DHbound}
Let $n,r,k\geq1$, and suppose $G$ is a tournament with at least $\frac{3}{2}n+\frac{3}{2}k-2r$ vertices and $T$ is an oriented forest with $n$ vertices, $r$ components and, in total, $k$ leaves and isolated vertices. Then, $G$ contains a copy of $T$.
\end{corollary}
\begin{proof} Label the components of $T$ as $T_1,\ldots,T_r$, and say, for each $i\in [r]$, that $T_i$ has $n_i$ vertices and, in total, $k_i$ isolated vertices and leaves. Note that $n_i+3k_i\geq 4$ for each $i\in [r]$. Take the largest $s\leq r$ for which there are vertex-disjoint subgraphs $S_i\subset G$, $i\in [s]$ such that, for each $i\in [s]$, $S_i$ is a copy of $T_i$. Suppose $s<r$, for otherwise we have already found a copy of $T$ in $G$, and note that
\[
\biggl|G-\bigcup_{i\in [s]}V(S_i)\biggr|\geq |G|-n+n_{s+1}\geq \sum_{i\in [r]\setminus \{s+1\}}\frac{n_i+3k_i-4}{2} +\frac{3n_{s+1}}{2}+\frac{3k_{s+1}}{2}-2\geq \frac{3n_{s+1}}{2}+\frac{3k_{s+1}}{2}-2.
\]
% \[
% |G-\cup_{i\in [s]}V(S_i)|\geq |G|-n+n_{s+1}\geq \sum_{i\in [r]\setminus \{s+1\}}\frac{n_i+3k_i-4}{2} +\frac{3n_{s+1}}{2}+\frac{3k_{s+1}}{2}-2\geq \frac{3n_{s+1}}{2}+\frac{3k_{s+1}}{2}-2.
% \]
Therefore, by Theorem~\ref{thm:DHbound}, $G-\cup_{i\in [s]}V(S_i)$ contains a copy of $S_{s+1}$, a contradiction.
\end{proof}
For Theorem~\ref{thm:fewcst} it is convenient to use the following bound, originally proved by El Sahili~\cite{El_S}, which can also be recovered from Theorem~\ref{thm:DHbound} by observing we must have $k\leq n-1$.

\begin{theorem}\label{cor:3n-3}
For each $n\geq 2$, every $(3n-3)$-vertex tournament contains a copy of every $n$-vertex oriented tree.
\end{theorem}

\subsection{Non-directed paths}\label{sec:andrewrulesok}
In both the proofs of Theorem~\ref{thm:fewlinear} and~\ref{thm:fewcst}, we will take a median order, $\sigma=v_1,\ldots, v_m$ say, of an $m$-vertex tournament, $G$ say, and carefully partition this order into intervals before embedding different parts of the tree into each interval. This embedding must thus work when $v_iv_j\in E(G)$ for each $1\leq i<j\leq m$, that is, when $G$ is a transitive tournament. Our embeddings then, will embed the vertices along a directed path into a consistent order under $\sigma$. From this, embedding directed paths will be more restrictive than embedding paths which have some changes of direction. Here, we will recall some results of Thomason which we use to embed paths with changes of directions, allowing us to assume later that each maximal bare subpath in the tree is directed.

To discuss the changes of direction in a path and recall these results, we use the terminology of blocks. A \emph{block} of an oriented path $P$ is a maximal directed subpath. When we introduce an oriented path we assume it has an associated overall direction, and thus a first and last vertex as well as a first block and a last block. When the path is a directed path we will always assume the associated direction is the natural one, i.e., the one in which the first vertex has no in-neighbours.

With only a couple of exceptions, when a tournament $G$ has one (or two) more vertices than an oriented path $P$, we can embed $P$ into $G$, while furthermore embedding one (or two) endvertices into a matching set of two vertices, if each endvertex is next to a block of length 1. That is, we have the following two results of Thomason.

\begin{theorem}[{\cite[Theorem~1]{THO}}]\label{thm:appending_non-directed_path}
Let $P$ be an oriented path of order $n$ with first block of length $1$. Let $G$ be a tournament of order $n+1$ and $X$ be a subset of $V(G)$ of order at least $2$. 

Then, there is a copy of $P$ in $G$ with first vertex in $X$.
\end{theorem}

\begin{theorem}[{\cite[Theorem~5]{THO}}]\label{thm:connecting_non-directed_paths}
Let $P$ be a non-directed oriented path of order $n$ with  first and last block of length 1. Let $G$ be a tournament of order $n+2$ and $X$ and $Y$ be two disjoint subsets of $V(G)$ of order at least 2.

If $P$ does not consist of three blocks with length one, then there is a copy of $P$ in $G$ with first vertex in $X$ and last vertex in $Y$.
\end{theorem}

\subsection{Short directed paths}
Having found parts of a tree in a median order of a tournament, we will often wish to join two of them with a directed path. The following lemma shows that this is possible across a median order, even in cases where the interval in between the vertices to be joined contains some forbidden vertices.

\begin{lemma}\label{lm:connecting_path_length_3}
Suppose $G$ is an $n$-vertex tournament  with a median order $\sigma=v_1,\ldots,v_n$. Then, for any set $A\subset V(G)\setminus\{v_1,v_n\}$ with $|A|\leq (n-8)/6$, there is a directed $v_1,v_n$-path in $G-A$ with length~3.
\end{lemma}
\begin{proof} If there are some distinct $x,y\in (N^+_G(v_1)\cap N^-_G(v_n))\setminus A$, then assume, by relabelling if necessary, that $xy\in E(G)$ and observe that $v_1xyv_n$ is a path with length 3 in $G-A$, as required. Therefore, suppose that $|(N^+_G(v_1)\cap N^-_G(v_n))\setminus A|\leqslant 1$.

By Lemma~\ref{lm:basic_properties_of_median_orders}~ii), we have $|N^+_G(v_1)\setminus \{v_n\}|,|N^-_G(v_n)\setminus\{v_1\}|\geqslant (n-2)/2$. Let $B_1=N^+_G(v_1)\setminus (A\cup N^-_G(v_n)\cup\{v_n\})$ and $B_2=N^-_G(v_n)\setminus (A\cup\{v_1\})$. Note that $|B_1|\geqslant n/2-2-|A|>0$ and $|B_2|\geqslant n/2-1-|A|$. Let $B_0=V(G)\setminus (B_1\cup B_2\cup \{v_1,v_n\})$, so that
\begin{equation}\label{eq:B0}
|B_0|=n-2-|B_1|-|B_2|\leqslant n-2-(n/2-2-|A|)-(n/2-1-|A|)=2|A|+1.
\end{equation}
Colour vertices in $B_0,B_1$ and $B_2$ respectively green, red and blue. If any blue vertex, $x$ say, has a red in-neighbour, $y$ say, then $v_1yxv_n$ is a path with length 3 in $G-A$, as required. Therefore, suppose that every in-neighbour of each blue vertex is a green vertex or a blue vertex, for otherwise we have the desired path.

Let $j$ be the largest integer such that $v_j$ is blue. Let $A_1=A\cap\{v_2,\ldots,v_{j-1}\}$ and $A_2=A\cap \{v_{j+1},\ldots,v_{n-1}\}$, so that $|A_1|+|A_2|=|A|$. For the appropriate $r$, let $I_1,\ldots,I_r$ be the maximal intervals of $v_2,\ldots,v_{j-1}$ consisting of only red and green vertices. Observe that, for each $i\in [r]$, the vertex after $I_i$ in $\sigma$ is blue, and has at least $|I_i|/2$ in-neighbours in $I_i$ by Lemma~\ref{lm:basic_properties_of_median_orders}~ii), all of which must be green. Thus, every interval $I_i$, $i\in [r]$, contains at least as many green vertices as red vertices.

As every red or green vertex before $v_j$ in $\sigma$ is in some interval $I_i$, $i\in [r]$, we have that there are at least as many green vertices as there are red vertices in $\{v_2,\ldots,v_{j-1}\}$.  As $|N^+_G(v_1)\cap \{v_2,\ldots,v_j\}|\geqslant (j-1)/2$ by Lemma~\ref{lm:basic_properties_of_median_orders}~ii), at least $(j-1)/2-|A_1|-1$ of the vertices in $\{v_2,\ldots,v_{j-1}\}$ are red. Therefore, there are at least $(j-1)/2-|A_1|-1$ green vertices in $\{v_2,\ldots,v_{j-1}\}$. By~\eqref{eq:B0} and the definition of $A_2$, we have that there at most $2|A|+1-|A_2|$ green vertices in $\{v_2,\ldots,v_{j-1}\}$. Thus, $2|A|+1-|A_2|\geqslant (j-1)/2-|A_1|-1$. Rearranging, and using that $|A_1|+|A_2|=|A|$, we get $3|A|\geqslant 2|A_2|+j/2-5/2$.

Now, by Lemma~\ref{lm:basic_properties_of_median_orders}~ii), $|N^-_G(v_n)\cap (\{v_{j+1},\ldots,v_{n-1}\})|\geqslant (n-1-j)/2$, so, as $v_j$ is the last blue vertex in $\sigma$, there are at least $(n-1-j)/2$ vertices in $A_2$.
Thus, $3|A|\geqslant 2|A_2|+j/2-5/2\geqslant (n-j)+j/2-7/2=n-j/2-7/2$. As $j\leqslant n-1$, we have $3|A|\geqslant (n-6)/2$, contradicting that $|A|\leq(n-8)/6$.
\end{proof}

\subsection{Trees and random sets}
Here we collect a number of elementary properties of oriented trees, which we use later, before recalling Chernoff's lemma. Our first proposition considers the number of maximal bare paths in a (non-oriented) tree with $k$ leaves, as follows.

\begin{proposition}\label{prop:deg3} An $n$-vertex tree $T$ with $k\geq 2$ leaves has at most $2k-3$ maximal bare paths, one of which must have length at least $(n-1)/(2k-3)$, and at most $2k-2$ vertices whose degree is not 2.
\end{proposition}
\begin{proof} For the appropriate $r$, let $P_1,\ldots,P_{r}$ be the maximal bare paths in $T$, and label vertices such that, for each $i\in [r]$, $P_i$ is an $x_i,y_i$-path. Note that the tree $T'$ formed from $T$ by replacing each path $P_i$, $i\in [r]$, by a single undirected edge has $r$ edges, $r+1$ vertices, $k$ leaves and no degree 2 vertices.
Therefore,
\[
2(|T'|-1)=2e(T')=\sum_{v\in V(T')}d_{T'}(v)\geqslant k+2(|T'|-k)+|\{v:d_{T'}(v)\geqslant 3\}|,
\]
and thus $|\{v:d_{T'}(v)\geqslant 3\}|\leqslant k-2$. As $\{v:d_{T}(v)\geqslant 3\}=\{v:d_{T'}(v)\geqslant 3\}$, $T$ has at most $2k-2$ vertices whose degree is not 2. Furthermore, $|T'|=r+1\leq k+(k-2)$, so that $r\leq 2k-3$. Finally, as $\sum_{i\in [r]}\ell(P_i)=e(T)=n-1$, one of the paths $P_i$, $i\in [r]$, has length at least $(n-1)/(2k-3)$.
\end{proof}

In the main embedding for both Theorem~\ref{thm:fewlinear} and Theorem~\ref{thm:fewcst}, we will embed collections of small subtrees with directed paths between them. The next two propositions (appropriately applied to an auxiliary oriented tree with vertices representing subtrees and edges representing paths) will give us an order in which these trees and paths will be embedded along a median order of the tournament.
We use Proposition~\ref{prop:treepart} for Theorem~\ref{thm:fewlinear}, and Proposition~\ref{prop:newgoodlabel} for Theorem~\ref{thm:fewcst}.

\begin{proposition}\label{prop:treepart} Every oriented tree $T$ with $n\geq 1$ vertices has a vertex partition $V(T)=V_1\cup\ldots \cup V_s$ of non-empty sets, for some $s\in [n]$, such that, for each edge $e\in E(T)$, for some $i\in [s-1]$, $e$ is an edge directed from $V_i$ to $V_{i+1}$.
\end{proposition}
\begin{proof} Noting that the statement is trivially true if $|T|\leqslant 2$, we prove this by induction on $|T|$. Suppose then it is true for all oriented trees with fewer than $n\geq 3$ vertices. We may assume, by directional duality, that $T$ has an out-leaf. Let $T'$ be formed from $T$ by removing such an out-leaf, $t$ say, and let $s\in [n-1]$ be such that there is a vertex partition $V(T')=V_1\cup\ldots \cup V_s$ of non-empty sets, such that, for each edge $e\in E(T')$, for some $i\in [s-1]$, $e$ is an edge directed from $V_i$ to $V_{i+1}$. Let $V_{s+1}=\emptyset$. Let $j$ be such that the in-neighbour of $t$ in $T$ is in $V_j$, and add $t$ to $V_{j+1}$. Taking the non-empty sets from $V_1,\ldots,V_{s+1}$ completes the proof of the inductive step, and hence the proposition.
\end{proof}

\begin{proposition}\label{prop:newgoodlabel}
Every $n$-vertex oriented tree $T$ has labellings $V(T)=\{t_1,\ldots,t_n\}$ and $E(T)=\{e_1,\ldots,e_{n-1}\}$, such that, for every $j\in[n-1]$, there is some $i_1,i_2\in [n]$ with $i_1\leqslant j<i_2$ and $e_j=t_{i_1}t_{i_2}$.
\end{proposition}

\begin{proof}
We proceed by induction on $n$, noting the proposition is trivial for $n=1$. For $n>1$, we may assume, by directional duality, that $T$ has an out-leaf. Let $t_n$ be this out-leaf, and $e_{n-1}$ its adjacent edge. By the inductive hypothesis, there are labellings $V(T-t_n)=\{t_1,\ldots,t_{n-1}\}$ and $E(T-t_n)=\{e_1,\ldots,e_{n-2}\}$, such that, for every $j\in[n-2]$, $e_j=t_{i_1}t_{i_2}$ for some $i_1\leqslant j<i_2$. Taking $V(T)=\{t_1,\ldots,t_n\}$ and $E(T)=\{e_1,\ldots,e_{n-1}\}$ completes the proof.
\end{proof}

Finally, in our embeddings we sometimes take small random sets, on which we use a standard Chernoff bound, as follows (see, for example~\cite{alon2004probabilistic}).

\begin{lemma}\label{chernoff} If $X$ is a binomial variable with standard parameters~$n$ and $p$, denoted $X=\mathrm{Bin}(n,p)$, and $\e$ satisfies $0<\e\leq 3/2$, then
\[
\P(|X-\E X|\geq \e \E X)\leq 2\exp\left(-\e^2\E X/3\right).\hfill\qedhere
\]
\end{lemma}
%%%%%%%%%%%%%%%%%%%%%%%%%%%%%%%%%%%%%%%%%%%%%%%%%%%%%%%%%%%%%%%%%%%%%%%%%%%%%%%%%%%%%%%%%%%%%%%%%%%%%%%%%%%%%%%%%%%%%%%%%%%%%%%%%%%%%%%%%%%%%%%%%%%%%%%%%%%%%%%%%%%%%%%%%%%%%%%%%%%%%%%%%%%%%%%%%%%%%%%%%%%%%%%%%%%%%%%%%%%%%%%%%%%%%%%%%%%%%%%%%%%%%%%%%%%%%%%%%%%%%%%%%%%%%%%%%%%%%%

\section{Proof of Theorem~\ref{thm:fewlinear}}\label{sect:fewlinear}
%We prove Theorem~\ref{thm:fewlinear}, that there exists some $C$ such that any oriented $n$-vertex tree $T$ with $k$ leaves is contained in any $(n+Ck)$-vertex tournament $G$.
In Section~\ref{subsec:reduction_to_bare_paths_directed}, we use the results quoted in Section~\ref{sec:andrewrulesok} to show that it is enough to prove Theorem~\ref{thm:fewlinear} in the case where all bare paths of $T$ are directed. That is, we reduce the proof to showing the following result.

\begin{theorem}\label{thm:fewlineardirected}
There is some $C>0$ such that each $(n+Ck)$-vertex tournament contains a copy of every $n$-vertex oriented tree with $k$ leaves in which every bare path is a directed path.
\end{theorem}

To prove Theorem~\ref{thm:fewlineardirected}, we first remove $O(k)$ long directed paths from $T$ to leave a forest with size linear in $k$. The components of this forest we embed into carefully chosen intervals of a median order with $O(k)$ spare vertices in total, using Corollary~\ref{cor:DHbound}. It remains then to embed the long directed paths, where we only have a constant number of spare vertices per path. This we do with Lemma~\ref{lm:many_connecting_paths} in Section~\ref{subsec:parallel_directed_paths}.
A simple modification of Dross and Havet's procedure for embedding arborescences into median orders (which they used to prove Theorem~\ref{thm:DH}) allows directed paths from specified first vertices to be embedded efficiently into a median order. To embed such paths with both endvertices specified, we adapt this procedure, using it to embed most of the directed paths, but, as soon as all but three edges of any path are embedded, using Lemma~\ref{lm:connecting_path_length_3} to connect the path to its desired last vertex. This allows us to find a set of directed paths while having only constantly many spare vertices per path (see Lemma~\ref{lm:many_connecting_paths}), which we use to prove Theorem~\ref{thm:fewlineardirected} in Section~\ref{subsec:proof_bare_paths_directed}.

\subsection{Reduction to trees with only directed bare paths}\label{subsec:reduction_to_bare_paths_directed}
 To prove Theorem~\ref{thm:fewlinear} from Theorem~\ref{thm:fewlineardirected}, we take a tree $T$, remove most of the middle section of the maximal bare paths with at least 6 blocks, and duplicate each new leaf created by this removal. (Here, a \emph{duplicated vertex} is a new vertex with exactly the same in- and out-neighbourbood as the matching original vertex.) Calling the resulting forest $T'$, if we have an embedding of $T'$ then the duplication of a leaf gives us two options to embed the original vertex from $T$. This will allow us to use the results in Section~\ref{sec:andrewrulesok} to embed the deleted path given enough other vertices in the tournament (with no further restriction on these other vertices). 
 
 Not every maximal bare path in $T'$ will be directed, but each such path will have at most 5 blocks. Adding a dummy leaf at any vertex in two blocks will give a forest $T''$ containing $T'$ whose maximal bare paths are all directed, allowing us to apply Theorem~\ref{thm:fewlineardirected} to each component. Importantly, $T'$, and hence $T''$, will still have $O(k)$ leaves.

\begin{proof}[Proof of Theorem~\ref{thm:fewlinear} from Theorem~\ref{thm:fewlineardirected}] Using Theorem~\ref{thm:fewlineardirected}, let $C\geqslant 8$ be large enough that, for every $\bar{n}$ and $\bar{k}$, every  $(\bar{n}+(C-8)\bar{k})$-vertex tournament contains a copy of every $\bar{n}$-vertex oriented tree with (at most) $9\bar{k}$ leaves in which every bare path is a directed path. Let $G$ be an $(n+Ck)$-vertex tournament, and let $T$ be an $n$-vertex oriented tree with $k$ leaves.

For the appropriate $r$, let $P_1,\ldots,P_{r}$ be the maximal bare paths in $T$, and label vertices such that, for each $i\in [r]$, $P_i$ is an $x_i,y_i$-path. By Proposition~\ref{prop:deg3}, we have $r\leqslant 2k-3$.
 Let $I\subset [r]$ be the set of $i\in [r]$ such that $P_i$ has at least 6 blocks.

For each $i\in I$, let $P_i^{(1)}$ and $P_i^{(2)}$ be the first two blocks of $P_i$ from $x_i$, and let $P_i^{(3)}$ and $P_i^{(4)}$ be the first two blocks of $P_i$ from $y_i$. Let $e_i^{(1)}$ be the furthest edge of $P_i^{(2)}$ from $x_i$ on $P_i$, and let $e_i^{(2)}$ be the furthest edge of $P_i^{(4)}$ from $y_i$ on $P_i$. Let $Q_i=(P_i-\sum_{j=1}^4 P_i^{(j)})+e_i^{(1)}+e_i^{(2)}$.

Note that, for each $i\in I$, the first and last block of $Q_i$ have length 1, its endvertices have degree 2 in $T$, and it has at least 4 blocks (and thus length at least 4). Label vertices so that, for each $i\in I$, $Q_i$ is a $u_i,v_i$-path. Let $T'$ be the forest formed from $T$ by, for each $i\in I$, deleting $Q_i$ and creating two new vertices, $u_i'$ and $v_i'$, so that $u_i'$ is a duplicate of $u_i$ and $v_i'$ is a duplicate of $v_i$. Note that $u_i,u_i',v_i$ and $v_i'$ are all leaves of $T'$.

Let $B$ be the set of vertices with degree 2 in $T'$ with either no in-neighbour or no out-neighbour, so that they lie in the intersection of two (consecutive) blocks. Observe that each such vertex must lie on some path $P_i$, $i\in [r]\setminus I$, or on $P_i^{(1)}\cap P_i^{(2)}$ or $P_i^{(3)}\cap P_i^{(4)}$ for some $i\in I$. Therefore, $|B|\leqslant 4(r-|I|)+2|I|$. Now, form $T''$ from $T'$ by taking each $v\in B$ and adding a new out-neighbour as a leaf, calling the new vertex $u_v$. We note here that all bare paths of $T''$ are directed paths.

Note that, if $\bar{T}$ is a component of $T'$, and $q$ is the number of paths $Q_i$ adjacent to $\bar{T}$ that are deleted when forming $T'$ from $T$, then $\bar{T}$ has at most $k-q+2q \leqslant k+|I|$ leaves. Furthermore, $T'$ has in total  $n+2|I|-\sum_{i\in I}(|Q_i|-2)\leqslant n+2|I|-3|I|=n-|I|$ vertices. Therefore, as $r\leqslant 2k-3$, each component of $T''$ has at most $k+|I|+|B|\leqslant 9k$ leaves and $T''$ in total has at most $n-|I|+|B|\leqslant n+8k$ vertices. Iteratively and vertex-disjointly, embed as many different components from $T''$ into $G$ as possible. If a component of $T''$, say a tree $\bar{T}$ with $\bar{n}$ vertices and $\bar{k}$ leaves,  is left unembedded then there are at least 
\[
|G|-(|T''|-|\bar{T}|)\geq (n+Ck)-(n+8k)+\bar{n}\geq \bar{n}+(C-8)k
\]
vertices not used in the embedding, and $\bar{k}\leq 9k$. Thus, by the choice of $C$, we can embed $\bar{T}$ using the unused vertices in $G$, a contradiction. Thus, $G$ contains a copy of $T''$, $S''$ say. 

For each $v\in B$, delete the copy of $u_v$ from $S''$, and let the resulting copy of $T'$ be $S'$. Note that, as $C\geqslant 8$ and $|I|\leqslant2k-3$,
\[
|V(G)\setminus V(S')|= n+Ck-|T'|=n+Ck-\Big(n+2|I|-\sum_{i\in I}(|Q_i|-2)\Big)\geqslant \sum_{i\in I}(|Q_i|-2),
\]
 and take vertex disjoint sets $A_i$, $i\in I$, in $V(G)\setminus V(S')$ with $|A_i|=|Q_i|-2$ for each $i\in I$.

 For each $i\in I$, let $\bar{u}_i,\bar{u}_i',\bar{v}_i,\bar{v}_i'$ be the copy of $u_i,u_i',v_i,v_i'$ respectively in $S'$. Using Theorem~\ref{thm:connecting_non-directed_paths}, for each $i\in I$, find a copy of $Q_i$, say $R_i$, in $G[A_i\cup\{\bar{u}_i,\bar{u}_i',\bar{v}_i,\bar{v}_i'\}]$ starting at $\bar{u}_i$ or $\bar{u}_i'$ and ending at $\bar{v}_i$ or $\bar{v}_i'$.
 Take then $S'$, and, for each $i\in I$, delete from $T'$ any vertices in $\{\bar{u}_i,\bar{u}_i',\bar{v}_i,\bar{v}_i'\}$ which are not an endvertex of $R_i$ and add the path $R_i$. Note that this gives a copy of $T$.
\end{proof}

\subsection{Joining vertex pairs with directed paths disjointly}\label{subsec:parallel_directed_paths}
We now connect multiple pairs of vertices with directed paths, where the start vertex for each path lies in a set $B_1$, and the end vertex lies in another set $B_2$, and the vertices of $B_1$ come before the vertices of $B_2$ in a median order. With Lemma~\ref{lm:connecting_path_length_3} we can find such paths; the challenge here is to find these paths when they collectively must use almost all of the intermediate vertices in the median order. To do this, we find most of the paths using a procedure of Dross and Havet~\cite{DRO-HAV} for embedding arborescences, modifying it with Lemma~\ref{lm:connecting_path_length_3} to attach each path to the correct end vertex when most of the path has been found.

\begin{lemma}\label{lm:many_connecting_paths}
Let $G$ be an $(m_0+m_1+m_2)$-vertex tournament, and suppose $\sigma=v_1,\ldots,v_{m_0+m_1+m_2}$ is a median order of $G$. Let $B_1\subseteq V(G)$ be the first $m_1$ vertices of $G$ according to $\sigma$, let $B_2\subseteq V(G)$ be the last $m_2$ vertices of $G$ according to $\sigma$, and let $B_0=V(G)\setminus(B_1\cup B_2)$. Let $(x_1,\ldots,x_r)\in B_1^r$ and $(y_1,\ldots,y_r)\in B_2^r$. For each $i\in [r]$, let $\ell_i\geqslant 5$.
Suppose finally that
\begin{equation}\label{eqn:m0bound}
m_0\geqslant m_1+m_2+\sum_{i\in [r]}\ell_i+22r-15.
\end{equation}
 Then, there are internally vertex-disjoint directed paths $P_1,\ldots,P_r$ in $G$ such that, for each $i\in [r]$, $P_i$ is a directed $x_i,y_i$-path with length $\ell_i$ and internal vertices in $B_0$.
\end{lemma}
\begin{proof}
Let $B_1'$ be the first $(m_1+2r-2)$ vertices of $B_0$ according to $\sigma$, and let $B_2'$ be the last $(m_2+2r-2)$ vertices of $B_0$ according to $\sigma$. Choose $X'=\{x_1',\ldots,x_r'\}\subseteq B_1'$ such that $x_i'\in N^+(x_i)$ for each $i\in[r]$. This is possible as, if for $i\in [r]$ we have chosen $x_1',\ldots,x_{i-1}'$, letting $U_i=\{w\in B_1:x_i<_\sigma w\leqslant_\sigma v_{m_1}\}$, then Lemma~\ref{lm:basic_properties_of_median_orders}~ii) gives
\begin{align*}
|N^+(x_i,B_1')\setminus\{x_1',\ldots,x_{i-1}'\}|&=|N^+(x_i,U_i\cup B_1')\setminus(U_i\cup\{x_1',\ldots,x_{i-1}'\})|\\
&\geqslant\frac{|U_i|+|B_1'|}{2}-|U_i|-|\{x_1',\ldots,x_{i-1}'\}|=\frac{|B_1'|-|U_i|}{2}-(i-1)\\
&\geqslant\frac{(m_1+2r-2)-(m_1-1)}{2}-(r-1)>0.
\end{align*}
Similarly, choose $Y'=\{y_1',\ldots,y_r'\}\subseteq B_2'$ such that $y_i'\in N^-(y_i)$ for each $i\in[r]$.

Let $A$ be a digraph formed by taking the disjoint union of directed paths $Q_i$, $i\in[r]$, where $Q_i$ has length $\ell_i-5$ for each $i\in [r]$. For $i\in[r]$, let $b_i$ be the first vertex and $c_i$ be the last vertex of $Q_i$. Note that $A$ has $\sum_{i\in[r]}(\ell_i-4)$ vertices.

Let $n_1=m_0-m_2-20r+13$. We now give a procedure which produces a partial embedding $\phi$ of $A$ into $G[\{v_{m_1+1},\ldots,v_{m_1+n_1}\}]$. Throughout, if a vertex $v_j$ of $G$ is the image of a vertex of $A$, we say that it is \emph{hit} and denote its pre-image by $a_j\in V(A)$. The sets $W_j$ record vertices of $G$ already used for the last two internal vertices of the paths $P_1,\ldots,P_r$ found by stage $j$.

\begin{itemize}
    \item Initially, set $W_{m_1+1}=\emptyset$ and $\phi(b_i)=x_i'$ for each $i\in[r]$ (so that $x_1',\ldots,x_r'$ are hit).
    \item For $j=m_1+1$ to $m_1+n_1$ in turn, do the following.
    \begin{enumerate}[label=(\alph*)]
      \item If $v_j$ is hit and $a_j=c_i$ for some $i\in[r]$, then, if possible, let $w_{i,1},w_{i,2}\in\{v_{j+1},\ldots,v_{m_1+m_0}\}\setminus (W_j\cup Y')$ be such that $w_{i,1}$ and $w_{i,2}$ are not yet hit, and $v_j\rightarrow w_{i,1}\rightarrow w_{i,2}\rightarrow y_i'$ in $G$. Set $W_{j+1}=W_j\cup\{w_{i,1},w_{i,2}\}$. If it is not possible to find such a $w_{i,1}$ and $w_{i,2}$, then simply set $W_{j+1}=W_j$.
      \item If $v_j$ is hit and $a_j\notin\{c_1,\ldots,c_r\}$, then extend $\phi$ if possible by assigning the first not-yet-hit out-neighbour of $v_j$ in $\{v_{j+1},\ldots,v_{m_1+n_1}\}\setminus W_j$ to the out-neighbour of $a_j$ in $A$. Set $W_{j+1}=W_j$.
    \item If $v_j\in W_j$, then set $W_{j+1}=W_j$.
    \item If $v_j\notin W_j$ and $v_j$ is not hit, then say that $v_j$ is \emph{failed}. Set $W_{j+1}=W_j$.
    \end{enumerate}
\end{itemize}

Note that, for each $m_1+1\leq j\leq m_1+n_1$, the vertices in $W_j$ are never hit, so that this procedure is well-defined.
We first show that the paths with length 3 in (a) are always found, as follows.
\begin{claim}\label{clm:wgood}
For each $m_1+1\leqslant j\leqslant m_1+n_1$, if $v_j$ is hit and $a_j=c_i$ for some $i\in[r]$, then the procedure finds vertices $w_{i,1}$ and $w_{i,2}$ as described in (a).
\end{claim}
\renewcommand\qedsymbol{$\boxdot$}
\begin{proof}[Proof of Claim~\ref{clm:wgood}] Suppose $j$ satisfies $m_1+1\leqslant j\leqslant m_1+n_1$, $v_j$ is hit and $a_j=c_i$ for some $i\in[r]$, so that, at stage $j$, we carry out (a). Let $s$ denote the number of times (a) was carried out before stage $j$.
As $W_j$ contains only vertices found in these previous instances of (a), we have $|W_j|\leq 2s$. 

At stage $j$, each path $Q_i$ has at most one vertex embedded by $\phi$ to $\{v_j,v_{j+1},\ldots,v_{m_1+n_1}\}$. Moreover, if a path $Q_i$ has a vertex embedded by $\phi$ to $\{v_{j+1},\ldots,v_{m_1+n_1}\}$, then (a) has not been carried out for that $c_i$. Thus, at most $r-1-s$ vertices in $\{v_{j+1},\ldots,v_{m_1+n_1}\}$ have been hit. Let $W'$ be the union of $W_j$, $Y'\setminus\{y_i'\}$, and the hit vertices in $\{v_{j+1},\ldots,v_{m_1+n_1}\}$.
Thus, as $s\leq r-1$,
\begin{equation}\label{eqn:Wbound}|W'|\leqslant 2s+(r-1)+ (r-1-s)\leqslant3(r-1).
\end{equation}

Let $j'$ be such that $v_{j'}=y_i'$, and note that, as $y_i'\in B_2'$, $j'\geqslant m_1+m_0-m_2-2r+3$, so that, as $n_1=m_0-m_2-20r+13$, we have
\begin{equation}\label{eq:int}
j'-j+1\geqslant m_1+m_0-m_2-2r+4-m_1-n_1= 18(r-1)+9\geqslant 6|W'|+9.%70(m_1+m_2+3r).
\end{equation}
Therefore, by Lemma~\ref{lm:connecting_path_length_3}, vertices $w_{i,1}$ and $w_{i,2}$ exist in $\{v_j,v_{j+1},\ldots,v_{j'}\}\setminus(W_j\cup Y')$ which have not yet been hit so that $v_j\to w_{i,1}\to w_{i,2}\to v_{j'}=y_i'$ in $G$.
\end{proof}

If the procedure finds a full embedding of $A$ into $G[\{v_{m_1+1},v_{m_1+2},\ldots,v_{m_1+n_1}\}]$, then observe that, for each $i\in [r]$, the image of $Q_i$ and the path $\phi(c_i)\rightarrow w_{i,1}\rightarrow w_{i,2}\rightarrow y_i'$ together give a path, $P_i'$ say, with length $\ell_i-2$ which is directed from $\phi(b_i)=x_i'$ to $y_i'$. Furthermore, the paths $P_i'$, $i\in [r]$, are vertex-disjoint with vertices in $B_0$. Taking $P_i$ to be the path $x_iP_i'y_i$ for each $i\in[r]$ gives the desired result.

All that remains to show is that the procedure produces a full embedding $\phi$ of $A$. Let $W=W_{m_1+n_1+1}$ and note that $|W|\leqslant 2r$. Assume for a contradiction that the procedure does not yield an embedding of $A$ into $G$. Then the set, $F$ say, of failed vertices in $\{v_{m_1+1},\ldots,v_{m_1+n_1}\}$ has
$|F|>n_1-|A|-|W|$.
Let $U\subseteq V(A)$ be the set of embedded vertices at the end of the procedure. Let $L$ be the set of vertices of $A$ which are the last embedded vertex on some path $Q_i$. Note we have $|L|=r$.

Say a vertex $a\in V(A)$ is \emph{active at stage $j$} if $\phi(a)\in\{v_{m_1+1},\ldots,v_{j-1}\}$ and $a$ has an out-neighbour $b$ that is not embedded in $\{v_{m_1+2},\ldots,v_j\}$ (i.e., either $b$ is not embedded or $\phi(b)\in\{v_{j+1},\ldots,v_{m_1+n_1}\}$). Now, if $v_j\in F$ comes before some vertex in $X'=\{x'_1,\ldots,x'_r\}\subset B_1'$, then it is possible there will be no active vertex at stage $j$. However, because we have assumed that the procedure does not yield an embedding of $A$ into $G$, if $v_j\in F$ and $j\geq 2m_1+2r-1$, then there must be some active vertex at stage $j$, for otherwise all the vertices of $A$ would be embedded in $\{v_{m_1+1},\ldots,v_{j-1}\}$.

Let $\bar{F}=\{v_j\in F:j\geq 2m_1+2r-1\}$, so that, for each $v_j\in \bar{F}$ we can define $r_j$ to be the largest index such that $a_{r_j}$ is active for $j$. Note, by the definition of an active vertex, $r_j<j$.
Furthermore, as $|F|>n_1-|A|-|W|$, $B'_1=\{v_{m_1+1},\ldots,v_{2m_1+2r-2}\}$ contains at least  $r$ vertices in the embedding (those in $X'$), and $|A|=\sum_{i\in [r]}(\ell_i-4)$, we have
\begin{equation}\label{eqn:Fbar}
|\bar{F}|>n_1-|A|-|W|-(m_1+2r-2-r)\geqslant m_0-m_2-20r+13-\sum_{i\in[r]}\ell_i+r-m_1+2\overset{\eqref{eqn:m0bound}}{\geqslant}3r.
\end{equation}

For each $v_j\in \bar{F}$, set $I_j=\{v_i:r_j<i\leqslant j\}$. We now bound from above the number of vertices of $\bar{F}$ in $I_j$, as follows.

\begin{claim}\label{clm:sizeIiF}
If $v_j\in \bar{F}$, then $|I_j\cap F|\leqslant |I_j\cap\phi(L)|+|I_j\cap W|$.
\end{claim}
\begin{proof}[Proof of Claim~\ref{clm:sizeIiF}] Let $J=(I_j\cap N^+(v_{r_j}))\setminus W$. As the out-neighbour of $a_{r_j}$ is never embedded in $I_j$, all the vertices in $J$ must be hit by the start of stage $r_j$. Thus, as $F\cap W=\emptyset$, we have $I_j\cap F\subseteq I_j\cap N^-(v_{r_j})$, so that
\begin{equation}\label{eq:interval_bound_1}
    |I_j\cap F|\leqslant |I_j\cap N^-(v_{r_j})|.
\end{equation}

Now, let $A_{r_j}$ and $A_{j-1}$ be the sub-digraphs of $G[v_{m_1+1},\ldots,v_j]$ which are the image of the partial embedding $\phi$ at the end of stage $r_j$ and stage $j-1$, respectively, restricted to the vertex set $\{v_{m_1+1},\ldots,v_j\}$. Observe the following.
\begin{itemize}
    \item Each vertex of $J$ is the last vertex of a path of $A_{r_j}$, as it is hit by the end of stage $r_j$ and occurs later in $\sigma$ than $r_j$.
    \item Any vertex in $I_j$ which is the last vertex of some path of $A_{j-1}$ must be the image of some $c_i$, for otherwise it is active for $j$, contradicting the definition of $r_j$. Thus, because $L$ is the set of vertices of $A$ which are the last embedded vertex on some path $Q_i$, such a vertex is in $I_j\cap \phi(L)$.
    \item As $r_j\leq j-1$, $A_{r_j}\subset A_{j-1}$, and $V(A_{j-1})\setminus V(A_{r_j})\subseteq I_j$, so $A_{j-1}$ must have at least as many paths terminating in $I_j$ as $A_{r_j}$ does.
\end{itemize}
Combining these three observations we have $|J|\leqslant |I_j\cap \phi(L)|$, and hence
\begin{equation}\label{eq:interval_bound_2}
    |I_j\cap N^+(v_{r_j})|\leqslant |I_j\cap\phi(L)|+|I_j\cap W|.
\end{equation}
Now, by Lemma~\ref{lm:basic_properties_of_median_orders}~ii), $|I_j\cap N^-(v_{r_j})|\leqslant|I_j\cap N^+(v_{r_j})|$. Together with (\ref{eq:interval_bound_1}) and (\ref{eq:interval_bound_2}), this proves the claim.
\end{proof}
\renewcommand\qedsymbol{$\square$}

Let $M$ be the set of indices $j$ such that $v_j\in\bar{F}$, and $I_j$ is maximal for inclusion among the sets $I_i$, with $v_i\in\bar{F}$. We will show that the sets $I_j$, $j\in M$ are disjoint.
If $i,j\in M$ with $i<j$ and $I_i\cap I_j\neq \emptyset$, then we have $r_j< i$. Observe that, as $a_{r_j}$ is active for $j$ and $\phi(a_{r_j})\in \{v_0,\ldots,v_{i-1}\}$, $a_{r_j}$ is also active for $i$, and hence $r_i\geqslant r_j$. Thus, $I_i\subset I_j$ and, as $i<j$, $I_i\neq I_j$, and hence $I_i$ is not maximal for inclusion among the sets $I_{i'}$, with $v_{i'}\in\bar{F}$, a contradiction.

Since $v_j\in I_j$ for all $v_j\in\bar{F}$, we have $\bar{F}\subseteq \cup_{j\in M} I_j$. As the sets $I_j$, $j\in M$, are pairwise disjoint, $|\bar{F}|\leqslant\sum_{j\in M}|I_j\cap F|$. By Claim~\ref{clm:sizeIiF}, we therefore obtain
\begin{equation*}
    |\bar{F}|\leqslant\sum_{j\in M}|I_j\cap F|\leqslant \sum_{j\in M}\left(|I_j\cap\phi(L)|+|I_j\cap W|\right)\leqslant|\phi(L)|+|W|\leqslant 3r,
\end{equation*}
contradicting~\eqref{eqn:Fbar}. This completes the proof of the lemma.
\end{proof}

\subsection{Proof of Theorem~\ref{thm:fewlineardirected}}\label{subsec:proof_bare_paths_directed}
Given Lemma~\ref{lm:many_connecting_paths} it is now straight-forward to prove Theorem~\ref{thm:fewlineardirected}. Given an $n$-vertex oriented tree $T$ with $k$ leaves whose maximal bare paths are directed, we label such paths with length at least 5 as $P_1,\ldots,P_r$, for the appropriate $r$ (which, by Proposition~\ref{prop:deg3}, satisfies $r=O(k)$). We can then consider $T$ to be formed of small vertex-disjoint subtrees $T_1,\ldots,T_{r+1}$ connected by the paths $P_1,\ldots,P_r$. We use Proposition~\ref{prop:treepart} to group these subtrees into classes, with the classes ordered so that each path $P_i$ goes from some class to the next class. Given then a tournament $G$ with $n+50k$ vertices, we divide a median order into intervals, with one interval for each class of subtrees and one for the set of paths between each pair of consecutive classes (see \eqref{eqn:splitorder}). Then, we then use Corollary~\ref{cor:DHbound} to embed the subtrees $T_i$ into their interval in the median order before using Lemma~\ref{lm:many_connecting_paths} to embed the paths $P_i$ with interior vertices in their interval in the median order.

\begin{proof}[Proof of Theorem~\ref{thm:fewlineardirected}] We will prove this with $C=50$, so let $\bar{n}=n+50k$.
Let $T$ be an $n$-vertex oriented tree with $k$ leaves in which every bare path is a directed path, and let $G$ be a $\bar{n}$-vertex tournament. Let $B$ be the set of vertices of $T$ which do not have degree 2, so that, by Proposition~\ref{prop:deg3}, $|B|\leqslant2k-2$.  Remove all maximal bare paths of length at least $5$ from $T$. Let $r$ be the number of removed paths, noting that, by Proposition~\ref{prop:deg3}, $r\leqslant 2k-3$, and label these paths as $P_1,\ldots,P_r$ (where we recall $\ell(P_i)$ denotes the length of $P_i$). Say the remaining forest $F$ has component trees $T_1,\ldots,T_{r+1}$, and, for each $i\in [r+1]$, let $k_i$ be the number of leaves of $T_i$ if $|T_i|\geq 2$, and let  $k_i=1$ if $|T_i|=1$. Note that $F$ is a union of $(|B|-1-r)$ maximal bare paths of $T$ with length at most 4 between vertices in $B$, resulting in a forest with $r+1$ components. Thus, we have that $|F|\leqslant|B|+3(|B|-1-r)\leqslant8k-3r-11$. Observing that every leaf or isolated vertex of $F$ is in $B$, we have $\sum_{i\in[r+1]}k_i\leqslant |B|\leq 2k-2$. We also note that
\begin{equation}
|F|=\sum_{i\in[r+1]}|T_i|
\;\;\text{ and }\;\;
    \sum_{i\in[r]}\ell(P_i)=|T|-|F|+r=n-\sum_{i\in[r+1]}|T_i|+r.\label{eqn:sumPi}
\end{equation}

Let $S$ be the oriented tree on vertex set $[r+1]$ with $ij\in E(S)$ whenever there is a directed path from $T_i$ to $T_j$ in $T$. By applying Proposition~\ref{prop:treepart} to $S$, let $s\leqslant r+1$ be such that there is a partition $I_1,\ldots,I_s$ of $[r+1]$ into non-empty sets such that, for each distinct $i,j\in [r+1]$, and $i'\in [s]$, if $i\in I_{i'}$ and there is a directed path from $T_i$ to $T_j$ in $T$, then $i'<s$ and $j\in I_{i'+1}$.
For each $i\in [s-1]$, let $J_i$ be the set of indices $j\in [r]$ such that $P_j$ is directed from $T_{i'}$ to $T_{j'}$ for some $i'\in I_i$ and $j'\in I_{i+1}$, and note that $\cup_{i\in [s-1]}J_i=[r]$.

Let $\sigma=v_1,\ldots, v_{\bar{n}}$ be a median order of $G$. In this median order take consecutive intervals
\begin{equation}\label{eqn:splitorder}
V_1,U_1,V_2,U_2,V_3,\ldots,V_{s-1},U_{s-1},V_s,
\end{equation}
appearing in that order, such that, for each $j\in [s]$,
\begin{align}
    |V_j|&=\Bigg\lceil\frac{3}{2}\sum_{i\in I_j}(|T_i|+k_i)\Bigg\rceil-2|I_j|\leq 
    \frac{3}{2}\sum_{i\in I_j}(|T_i|+k_i)+\frac{1}{2}-2|I_j|,\label{eqn:Ui}
    \end{align}
    and, for each $j\in [s-1]$,
\begin{align}\label{eqn:Vj}
|U_j|&=|V_j|+|V_{j+1}|+\sum_{i\in J_j}\ell(P_i)+22|J_j|-15.
\end{align}
Note that this is possible, as
\begin{align*}
    \sum_{j=1}^s|V_j|+\sum_{j=1}^{s-1}|U_j|&\overset{\eqref{eqn:Vj}}{\leqslant}
     3\sum_{j=1}^s|V_j|+\sum_{j\in [r]}\ell(P_j)+22\sum_{j\in [s-1]}|J_j|-15(s-1)
    \\
    &\overset{\eqref{eqn:Ui}}{\leq} \frac{9}{2}\sum_{i\in [r+1]}(|T_i|+k_i)+\frac{3}{2}s-6\sum_{j=1}^s|I_j|+\sum_{j\in [r]}\ell(P_j)+22r-15(s-1)\\
    &\overset{\eqref{eqn:sumPi}}{\leq}n+r+\frac{7}{2}|F|+\frac{9}{2}\sum_{i\in[r+1]}k_i-6(r+1)+22r\\
    &\leqslant n+\frac{7}{2}(8k-3r-11)+\frac{9}{2}(2k-2)+17r-6\\
    &\leqslant n+37k+\frac{13}{2}r\leqslant n+50k,
\end{align*}
where we have used that $r\leq 2k-3$.
By Corollary~\ref{cor:DHbound} and \eqref{eqn:Ui}, a copy of $\cup_{i\in I_j}T_i$ exists in $G[V_j]$ for each $j\in [s]$. By Lemma~\ref{lm:many_connecting_paths} and \eqref{eqn:Vj}, for each $j\in [s-1]$, the $|J_j|$ paths $P_i$, $i\in J_j$, between $\cup_{i\in I_j}T_i$ and $\cup_{i\in I_{j+1}}T_i$ can then be embedded in the intervals $V_j,U_j,V_{j+1}$ with the appropriate first and last vertex in $V_j$ and $V_{j+1}$, respectively, and internal vertices in $U_j$. This completes the embedding of $T$, and hence the proof of the theorem.
\end{proof}

\section{Proof of Theorem~\ref{thm:fewcst}}\label{sect:fewcst}
As an illustrative case, let us first sketch Theorem~\ref{thm:fewcst} for trees consisting of a directed path between two arborescences, as follows.
Suppose we have a directed path $P$, an in-arborescence $S$ with root the first vertex of $P$, and an out-arborescence $S'$ with root the last vertex of $P$, and suppose that $S\cup P\cup S'$ is an oriented tree with $n$ vertices. Say $S$ has $k$ in-leaves and $S'$ has $k'$ out-leaves, and the tournament $G$ has $m:=n+k+k'-2$ vertices and a median order $v_1,\ldots,v_m$. Using Lemma~\ref{lm:basic_properties_of_median_orders}~i) and Theorem~\ref{thm:DH} (via directional duality), we can embed $S$ into $G[\{v_1,\ldots,v_{|S|+k-1}\}]$ with the root vertex embedded to $v_{|S|+k-1}$. Similarly, we can embed $S'$ into $G[\{v_{m-|S'|-k'+2},\ldots,v_m\}]$ with the root vertex of $S'$ embedded to $v_{m-|S'|-k'+2}$. Finally, by Lemma~\ref{lm:basic_properties_of_median_orders}~ii), we have $v_{|S|+k-1}\to v_{|S|+k}\to \ldots \to v_{m-|S'|-k'+2}$, so we can use this path to embed the $n-|S|-|S'|+2=m-|S|-|S'|-k-k'+4$ vertices of $P$ and complete an embedding of $T$ into $G$.

Essentially, all our embeddings will look like this, where $P$ will be a very long path, but with some additional subtrees and paths found within the interval we use to embed $P$. For example, suppose now the tree $T$ also has a subtree $F$ which shares one vertex, $t$ say, with $S$, where $t$ only has out-neighbours in $F$. If $P$ is a long path (compared to $|F|,|S|,|S'|$) then we can embed $T=F\cup S\cup P\cup S'$ into a tournament $G$ with $m:=|T|+k+k'-2$ vertices as follows. Carry out the above embedding of $S$ and $S'$ into the start and end respectively of a median order $v_1,\ldots,v_m$ of $G$ and note that the path $Q:=v_{|S|+k-1}\to v_{|S|+k}\to \ldots \to v_{m-|S'|-k'+2}$ has $|F|-1+|P|$ vertices. If $s$ is the embedding of $t\in V(S)$, then by Lemma~\ref{lm:basic_properties_of_median_orders}~ii) and as $|Q|\geq |P|-1\gg |F|,|S|$, $s$ will have many out-neighbours in this path, enough that we can easily embed $F-t$ among the out-neighbours of $s$ in $Q$ (using, in particular, Corollary~\ref{cor:DHbound}). However, we wish to do this so that there is a directed path between $v_{|S|+k-1}$ and $v_{m-|S'|-k'+2}$ covering exactly the $|Q|-(|F|-1)=|P|$ vertices of $V(Q)$ which are not used to embed $F-t$. 

To do this, before embedding $F$, we first find a short directed $v_{|S|+k-1},v_{m-|S'|-k'+2}$-path $R$ with vertices in $V(Q)$ so that every vertex in $V(Q)$ has at least one out-neighbour on $R$ occurring after some in-neighbour on $R$. The path $R$ will be short enough that we can embed $F-t$ in the out-neighbours of $s$ in $V(Q)$ while avoiding $V(R)$. Once $F-t$ has been embedded, we slot the remaining vertices in $V(Q)$ into $R$ one by one. This will be possible from the property of $R$ as we are working in the tournament $G$ (see Claim~\ref{clm:Rismall}). Note that, in the language of absorption (as codified by R\"odl, Ruci\'nski and Szemer\'edi~\cite{rodl2006dirac}), $R$ is a path which can absorb any set of vertices from the interval of the median order between its first and last vertex.

More generally, we can embed small trees attached with an out-edge from $S\cup P\cup S'$, as long as the attachment point is not too late in $P$, and also not in $S'$, by embedding such small trees within the interval for the path $P$. Similarly, we can embed small trees attached with an in-edge from $S\cup P\cup S'$, as long as the attachment point is not too early in $P$, and also not in $S$. We can also use Lemma~\ref{lm:connecting_path_length_3} to add short paths between vertices in the interval from $P$ that are not too close together. We therefore decompose any $n$-vertex tree $T$ with $k$ leaves by finding a digraph $D$ which can be built in this way and which contains $T$.

Roughly speaking, we call the digraph $D$ a \emph{good decomposition for $T$} if it contains $T$ and can be built from some $S\cup P\cup S'$ as described above by adding digraphs in these ways; this is defined precisely in Section~\ref{subsec:gooddecomp}. In Section~\ref{subsec:finding_a_good_decomposition}, we show that there is a good decomposition for any tree without a subpath that we could otherwise deal with using Section~\ref{sec:andrewrulesok} as before. Then, in Section~\ref{subsec:embedding_a_good_decomposition}, we show it is possible to embed any good decomposition of any $n$-vertex tree with $k$ leaves into an $(n+k-2)$-vertex tournament. Finally, in Section~\ref{subsec:final}, we put this together to prove Theorem~\ref{thm:fewcst}.

\subsection{\texorpdfstring{$(r,m)$}{(r,m)}-good decompositions}\label{subsec:gooddecomp}
We now define a good decomposition precisely, using the follow definition of a path partition.

\begin{defn} Say a sequence of paths $P_1\ldots P_\ell$ is a \emph{path partition of a path $P$} if $P=\cup_{i\in [\ell]}P_i$ and, for each $i\in [\ell-1]$, the end vertex of $P_{i}$ is the start vertex of $P_{i+1}$, and all the paths are otherwise pairwise vertex disjoint.
\end{defn}

Roughly speaking, as depicted in Figure~\ref{fig:gooddecomp}, an $(r,m)$-good decomposition for a tree $T$ is a digraph $D$ with $T\subseteq D$, such that $D$ can be constructed by taking a long directed path $P$ from the root of an in-arborescence $S_1$ to the root of an out-arborescence $S_{r+1}$, attaching small forests $F_i$ to a limited number of well-separated subpaths $S_i$ of $P$, and, finally, attaching short directed paths $Q_i$ between some of these well-separated subpaths and forests. More precisely, we define an $(r,m)$-good decomposition as follows.

\begin{defn}\label{def:gooddecomp}
Say that a digraph $D$ is an \emph{$(r,m)$-good decomposition} for an $n$-vertex oriented tree $T$ if $V(D)=V(T)$, and, for some distinct $x,y\in V(D)$, there is a directed $x,y$-path $P$ with path partition
\begin{equation}\label{eqn:Ppart}
P=P_1S_2P_2S_3\ldots P_{r-1}S_{r}P_{r},
\end{equation}
an in-arborescence $S_1$ with root $x$, an out-arborescence $S_{r+1}$ with root $y$, and
\begin{itemize}
\item forests $F_i^+$, $F_i^-$, $i\in [r+1]$, and
\item for some $0\leq \ell \leq 2r$, vertices $s_i,t_i$ and directed $s_i,t_i$-paths $Q_i$, $i\in [\ell]$,
\end{itemize}
such that, letting $F_i=F_i^-\cup F^+_i$ for each $i\in [r+1]$, the following hold.
\stepcounter{capitalcounter}
\begin{enumerate}[label = {\bfseries \Alph{capitalcounter}\arabic{enumi}}]
\item $T\subseteq S_1\cup P\cup S_{r+1}\cup (\cup_{i\in [r+1]} F_i)\cup (\cup_{i\in [\ell]}Q_i)=D$.\label{prop1}
\item The following sets, over $i\in [r+1]$ and $j\in [\ell]$, form a partition of $V(T)=V(D)$:\label{prop2}
\[
V(P),\;V(F_i^+)\setminus V(S_i),\; V(F_i^-)\setminus V(S_i),\; V(S_1)\setminus \{x\},\; V(S_{r+1})\setminus \{y\},\; V(Q_j)\setminus \{s_j,t_j\}.
\]
\item For each $i\in [r]$, $P_i$ has length at least $2000m$.\label{prop3}
\item For each $i\in [r+1]$ and $\diamond\in \{+,-\}$, $V(S_i)\subset V(F_i^\diamond)$, $|F_i^\diamond|\leq m$, and $F_i^\diamond$ is a forest in which each component has exactly one vertex in $S_i$, which furthermore has only $\diamond$-neighbours in $F_i^\diamond$.\label{prop4}
\item $E(F_1^-)=E(F_{r+1}^+)=\emptyset$ and $|S_1|,|S_{r+1}|\geq 2$.\label{prop5}
\item The total number of in-leaves of $S_1$ and out-leaves of $S_{r+1}$ is at most the number of leaves of $T$.\label{prop6}
\item For each $i\in [\ell]$, one of the following holds.\label{prop78}
\begin{enumerate}[label = {\bfseries \Alph{capitalcounter}\arabic{enumi}.\arabic{enumii}}]
\item For some $1\leq j<j'\leq r+1$, $Q_i$ is a directed path from $F_j$ to $F_{j'}$ with length $3(j'-j)+1$.\label{prop7}
\item For some $2\leq j\leq r$, $Q_i$ is a directed path with length 3 from $V(F_j^-)\setminus V(S_j)$ to the last vertex of $S_j$.\label{prop8}
\end{enumerate}
\end{enumerate}
\end{defn}

\begin{figure}
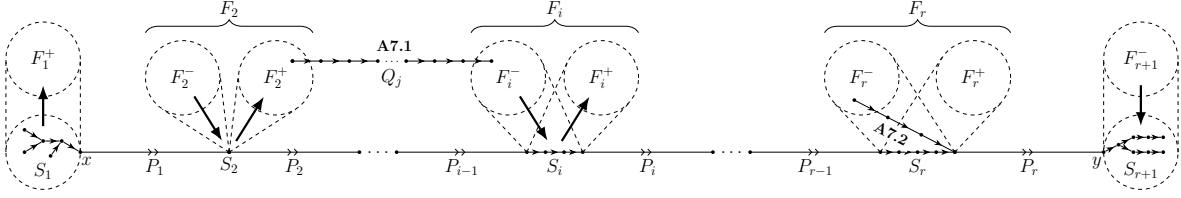

\DiagramGoodDecomposition
\vspace{-0.75cm}
\caption{An $(r,m)$-good decomposition.\label{fig:gooddecomp}}
\end{figure}
%\DiagramC

%%%%%%%%%%%%%%%%%%%%%%%%%%%%%%%%%%%%%%%%%%%%%%%%%%%%%%%%%%%%%%%%%%%%%%%%%%%%%%%%%%%%%%%%%%%%%%%%%%%%%%%%%
%%%%%%%%%%%%%%%%%%%%%%%%%%%%%%%%%%%%%%%%%%%%%%%%%%%%%%%%%%%%%%%%%%%%%%%%%%%%%%%%%%%%%%%%%%%%%%%%%%%%%%%%

\subsection{Finding a good decomposition}\label{subsec:finding_a_good_decomposition}
As noted before, by the results in Section~\ref{sec:andrewrulesok}, we will be able to assume that our $n$-vertex tree $T$ with $k$ leaves in Theorem~\ref{thm:fewcst} mostly consists of directed bare paths. To find a good decomposition, we consider these paths and arrange them in order of decreasing length. Identifying a point where the length of these paths drops significantly (perhaps including all the paths), we show that removing these long paths creates a forest in which each component is much smaller than each of the removed paths. Next, we order these paths and components using Proposition~\ref{prop:newgoodlabel}. Taking (essentially) the removed paths as the paths $P_i$, carefully chosen directed subpaths $S_i$ of the components of the forest (see \ref{case1}--\ref{case4}) and some dummy edges if necessary will form the path in~\eqref{eqn:Ppart}. After the careful selection in \ref{case1}--\ref{case4}, we will be able to divide naturally the rest of $T$ into the other sets in the decomposition.

\begin{lemma}\label{lem:gooddecomp} Let $1/n \ll \mu\ll 1/k$. Let $T$ be an $n$-vertex oriented tree with $k\geq 2$ leaves and no bare path of length at least $\mu n$ with first and last block of length 1 and whose endvertices have degree 2 in $T$. Then, for some $r\leq 10k$ and $m\geq \mu n$, $T$ has an $(r,m)$-good decomposition.
\end{lemma}
\begin{proof}
We will construct an $(r,m)$-good decomposition using the notation in Definition~\ref{def:gooddecomp}, and confirm that each of \ref{prop1}--\ref{prop78} hold.

Let $p$ be the number of maximal bare paths of $T$, and let them be $T'_1,\ldots,T'_{p}$. By Proposition~\ref{prop:deg3}, we have $p\leq 2k-3$. Observe that each $T'_i$ has fewer than $\mu n$ edges that are not contained in the first two blocks or the last two blocks, for otherwise, taking the last edge of the second block, and the first edge of the penultimate block, and all the edges between them on $T'_i$, gives a bare path with length at least $\mu n$ with first and last block of length 1 whose endvertices have degree 2 in $T$.
Let $q$ be the number of maximal directed bare paths of $T$ with length at least $\mu n$, and let them be $T_1,\ldots,T_q$ with length $\ell_1,\ldots,\ell_q$ respectively, so that $\ell_1\geq \ell_2\geq \ldots \geq \ell_q$. By the above observation, we find $q\leqslant 4p\leqslant 8k-12$, and $|T-T_1-\ldots-T_q|\leqslant(2k-3)(4\mu n+\mu n)\leqslant 10k\mu n$. Furthermore, as $\mu\ll 1/k$, we must have that $q\geq 1$ and $\ell_1\geq n/2q\geq n/20k$.

Now, let $r\in [q-1]$ be the smallest integer such that $\ell_r> 10^6k\ell_{r+1}$, if it exists, and $r=q$ otherwise.  Let $m=\ell_r/2500$. Note that, as $\ell_1\geq n/20k$ and $\mu\ll 1/k$,
\begin{equation}\label{eqn:m}
m\geq \frac{\ell_1}{2500\cdot (10^6k)^{r-1}}\geq \frac{n/20k}{2500\cdot (10^6k)^{8k-12}}\geq\mu n.
\end{equation}
Note that, as $r\leq q$, $r\leq 10k$ and $m\geq \mu n$, as required.
As $T-T_1-\ldots-T_r$ is the union of $T-T_1-\ldots-T_q$ and at most $8k-12$ paths of length at most $\ell_r/10^6k$, we have $|T-T_1-\ldots-T_r|\leqslant 10k\mu n+m/4\leqslant m/2$. Note that $T-T_{1}-\ldots-T_r$ has $r+1$ components. Say these are $R_1,\ldots,R_{r+1}$, and note that $|R_i|\leqslant |T-T_1-\ldots-T_r|\leqslant m/2$ for each $i\in[r+1]$.

Using Proposition~\ref{prop:newgoodlabel}, relabel the components $\{R_1,\ldots,R_{r+1}\}$ and paths $\{T_1,\ldots,T_r\}$, and define functions $i^-,i^+:[r]\to[r+1]$, so that, for every $j\in[r]$, $T_j$ is a directed path from $R_{i^-(j)}$ to $R_{i^+(j)}$, and $i^-(j)\leqslant j<i^+(j)$.

For each $j\in[r]$, label vertices so that $T_j$ is an $x_j',y_j'$-path directed from $x_j'\in V(R_{i^-(j)})$ to $y_j'\in V(R_{i^+(j)})$. Let $I\subseteq \{2,\ldots,r\}$ be the set of $i$ with $y_{i-1}'\in V(R_i)$, $x_i'\in V(R_i)$, and such that the path between $y_{i-1}'$ and $x_i'$ in $R_i$ is not directed from $y_{i-1}'$ to $x_i'$. For each $j\in[r]$, let $Q_j^+$ be the path consisting of the last $3(i^+(j)-j-1)+1\geq 1$ edges of $T_j$. For each $j\in[r]\setminus I$, let $Q_j^-$ be the path consisting of the first $3(j-i^-(j))+1\geq 1$ edges of $T_j$. For each $j\in I$, let $Q_j^-$ be the path consisting of the first 3 edges of $T_j$. Note that the lengths of the paths $Q_j^+,Q_j^-$ are always much smaller than the length of the path $T_j$.

For each $i\in[r]$, let $P_i$ be such that $T_i=Q^-_iP_iQ^+_i$ is a path partition. Label vertices so that $P_i$ is an $x_i,y_i$-path directed from $x_i$ to $y_i$. Note that each path $P_i$ is $T_i$ with up to $3r+1$ edges removed from each end. As the original length of such a path was at least $\ell_r=2500m$, and we have $1/n\ll\mu\ll 1/r$, we have by \eqref{eqn:m} that \ref{prop3} holds.

Let $x=x_1$ and note that $Q_1^-=x_1'x$. Let $S_1\subset R_1+x_1'x$ be the maximal in-arborescence in $R_1+x_1'x$ with root $x$. Note we have that $|S_1|\geq 2$. Let $y=y_r$ and note that $Q_r^+=yy_r'$. Let $S_{r+1}$ be the maximal out-arborescence in $R_{r+1}+yy_r'$ with root $y$. Note we have $|S_{r+1}|\geq 2$.

If $k_0$ is the number of in-leaves of $S_1$, then as its root $x$ is an out-leaf, $S_1$ has $k_0+1$ leaves. Similarly, if $k_1$ is the number of out-leaves of $S_{r+1}$, then $S_{r+1}$ has $k_1+1$ leaves. Now, take the path, $S$ say,  between $S_1$ and $S_{r+1}$ in $T$ and note that the tree $S_1\cup S\cup S_{r+1}$ has $(k_0+1)+(k_1+1)-2=k_0+k_1$ leaves. Noting that $T$ has at least as many leaves as $S_1\cup S\cup S_{r+1}\subset T$ completes the proof that \ref{prop6} holds.

Now, for each $i\in \{1,r+1\}$ and each $\diamond\in \{+,-\}$, let $F_{i}^\diamond\subset S_{i}\cup R_i$ be the digraph formed from the union of the paths in $(S_{i}\cup R_i)-E(S_i)$ from $V(S_i)$ which start with a $\diamond$-edge, and let $F_{i}=F_{i}^+\cup F^-_{i}=(S_{i}\cup R_i)-E(S_{i})$. Note that, by the maximality of $S_1$ as an in-arborescence and the maximality of $S_{r+1}$ as an out-arborescence, we have that $E(F_1^-)=E(F_{r+1}^+)=\emptyset$, completing the proof that \ref{prop5} holds. For each $i\in \{1,r+1\}$, $|F_i|\leq |R_i|+1\leq m/2+1\leq m$, so \ref{prop4} holds as well for $i\in \{1,r+1\}$.

We now construct $y_{i-1},x_i$-paths $S_i$, for each $2\leqslant i\leqslant r$. For each such $i$, we consider $Q^+_{i-1}\cup R_i\cup Q^-_i$, and add up to two edges (according to the cases below) before finding a directed path $S_i$ through the resulting digraph. We divide into cases \ref{case1}--\ref{case4} according to whether $y_{i-1}'\in V(R_i)$ (i.e., if $i^+(i-1)=i$) and whether $x_i'\in V(R_i)$ (i.e., if $i^-(i)=i$) . These cases are depicted in Figure~\ref{fig:paths-through-islands}. Note that, if $y'_{i-1}\in V(R_i)$ then $Q_{i-1}^+$ consists of only the edge $y_{i-1}y_{i-1}'$, and if $x_i'\in V(R_i)$ with $i\notin I$, then $Q_i^-$ consists of only the edge $x_i'x_{i}$.
Precisely, for each $2\leq i\leq r$, we do the following.

\stepcounter{capitalcounter}
\begin{enumerate}[label={\bfseries\Alph{capitalcounter}\arabic{enumi}}]
\item If $y_{i-1}'$ and $x_i'$ are both in $V(R_i)$, then do the following.\label{case1}
\begin{enumerate}[label = {\bfseries \Alph{capitalcounter}\arabic{enumi}.\arabic{enumii}}]
\item\label{case11} If the $y_{i-1}',x_{i}'$-path in the tree $R_i$ is a directed path from $y_{i-1}'$ to $x_{i}'$, then let $S_{i}$ be the directed path from $y_{i-1}$ to $x_i$ in $R_i+y_{i-1}y_{i-1}'+x_i'x_i$.
\item\label{case12} If the $y_{i-1}',x_{i}'$-path in the tree $R_i$ is not a directed path from $y_{i-1}'$ to $x_{i}'$ (i.e., if $i\in I$), then let $S_{i}'$ be the maximal directed subpath from $y_{i-1}'$ that it contains. Let $S_i$ be the path consisting of the edge $y_{i-1}y_{i-1}'$, followed by $S_i'$, followed by a new edge from the endvertex of $S_i'$ to $x_i$.
\end{enumerate}
\item\label{case2} If $y_{i-1}'\in V(R_i)$ and $x_i'\notin V(R_i)$, then let $S_i$ be the path $y_{i-1}y_{i-1}'x_i$.
\item\label{case3} If $y_{i-1}'\notin V(R_i)$ and $x_i'\in V(R_i)$, then let $S_i$ be the path $y_{i-1}x_i'x_i$.
\item\label{case4} If $y_{i-1}',x_i'\notin V(R_i)$, then let $z\in V(R_i)$ be arbitrary, and let $S_i$ be the path $y_{i-1}zx_{i}$.
\end{enumerate}

\begin{figure}
\DiagramPathsThroughIslands
\vspace{-0.75cm}
\caption{Cases~\ref{case1}-\ref{case4}.}\label{fig:paths-through-islands}
\end{figure}

Now, for each $2\leq i\leq r$, we choose $F_{i}^+$, $F_{i}^-$ and $F_{i}=F_{i}^+\cup F_{i}^-$. To do so, for each $2\leq i\leq r$ and each $\diamond\in \{+,-\}$, let $F_{i}^\diamond\subset S_{i}\cup R_i$ be the digraph formed from the union of the paths in $(S_{i}\cup R_i)-E(S_{i})$ from $V(S_{i})$ which start with a $\diamond$-edge, and let $F_{i}=F_{i}^+\cup F^-_{i}=(S_{i}\cup R_i)-E(S_{i})$.
Note that $F_i^+$ and $F_i^-$ could consist of a single vertex. For each $2\leq i\leq r$, $|F_{i}|= |R_i|+2\leq m/2+2\leq m$. We now have that \ref{prop4} holds for each $i\in[r+1]$, as required.

Let $\ell$ be the number of paths $Q_i^\diamond$, $i\in[r]$, $\diamond\in\{+,-\}$ with length greater than 1, so that $0\leqslant\ell\leqslant2r$. Relabel these paths arbitrarily as $Q_i,i\in[\ell]$. Note that, as we created no new vertices, we have that $V(D)\subset V(T)$ (with equality once we confirm $T\subset D$ below). It is left then to prove that \ref{prop1}, \ref{prop2}, and \ref{prop78} hold and check the properties at the start of Definition~\ref{def:gooddecomp}.

Note that, for each $2\leq i\leq r$, $S_i$ was a directed $y_{i-1},x_i$-path. Therefore, as $x=x_1$ and $y=y_r$,
\begin{equation}
P:=P_1S_2P_2S_2\ldots P_{r-1}S_{r}P_{r}
\end{equation}
is a path partition of the directed $x,y$-path $P$. Furthermore, we have that $S_1$ is an in-arborescence with root $x$ and that $S_{r+1}$ is an out-arborescence with root $y$.

Now,  by construction, $T\subseteq P\cup S_1\cup S_{r+1}\cup (\cup_{i\in [r+1]}F_i)\cup(\cup_{i\in [r],\diamond\in\{+,-\}}Q_i^\diamond)=D$. Whenever $Q_i^+$ has length~1 and $i<r$, we have that $i^+(i)=i+1$, so $S_{i+1}$ is chosen in \ref{case11}, \ref{case12}, or \ref{case2}, and hence $Q_i^+=y_{i-1}y_{i-1}'\subseteq S_{i+1}$. Note that $Q_r^+$ has length~1, and $Q_r^+=yy_r'$ is in $S_{r+1}$. Whenever $Q_i^-$ has length 1 and $i>1$, we must have that $i\notin I$ and $i^-(i)=i$, and therefore $S_{i}$ is chosen in \ref{case11} or \ref{case3}, so that $Q_i^-=x_i'x_i\subseteq S_i$. 
Note that $Q_1^-$ has length 1, and $Q_1^-=x_1'x$ is in $S_1$.
Therefore, $P\cup(\cup_{i\in [r],\diamond\in\{+,-\}}Q_i^\diamond)=P\cup(\cup_{i\in[\ell]}Q_i)+x_1'x+yy_r'$, and so $T\subseteq P\cup S_1\cup S_{r+1}\cup (\cup_{i\in [r+1]}F_i)\cup(\cup_{i\in[\ell]}Q_i)=D$ and \ref{prop1} holds. 

Furthermore, note that $V(R_i)$, $i\in [r+1]$, and $V(T_i)\setminus\{x_i',y_i'\}$, $i\in [r]$, form a partition of $V(T)$. For each $i\in [r]$, $V(Q_i^-)\setminus \{x_i,x_i'\}$, $V(P_i)$ and $V(Q_i^+)\setminus \{y_i,y_i'\}$ form a partition of $V(T_i)\setminus\{x_i',y_i'\}$. For each $2\leq i\leq r$, by the choice of $F^+_i$ and $F^-_i$, $V(F_i^+)\setminus V(S_i)$, $V(F_i^-)\setminus V(S_i)$ and $V(S_i)\setminus \{y_{i-1},x_i\}$ form a partition of $R_i$, while $V(F_1^-)\setminus V(S_1)=\emptyset$, $V(F_1^+)\setminus V(S_1)$ and $V(S_1)\setminus \{x_1\}$ partition $V(R_1)\setminus \{x_1\}$, and $V(F_{r+1}^-)\setminus V(S_{r+1})$, $V(F_{r+1}^+)\setminus V(S_{r+1})=\emptyset$ and $V(S_{r+1})\setminus \{y_{r}\}$ partition $V(R_{r+1})\setminus \{y_r\}$. As $V(P)=(\cup_{i\in [r]}V(P_i))\cup (\cup_{2\leq i\leq r}(V(S_i)\setminus \{y_{i-1},x_i\}))$, the sets listed in \ref{prop2} form a partition of $V(T)$.

Therefore, we need only show that, for each path $i\in [\ell]$, either \ref{prop7} or \ref{prop8} holds. If $Q_i=Q_j^+$ for some $j\in[r]$, then $Q_i$ is a directed $y_j,y_j'$-path of length $3(i^+(j)-(j+1))+1>1$, so that $i^+(j)>j+1$. As $y_j\in V(S_{j+1})\subseteq V(F_{j+1})$ and $y_j'\in V(R_{i^+(j)})\subseteq V(F_{i^+(j)})$,  \ref{prop7} holds for $Q_i$. If $Q_i=Q_j^-$ for some $j\in[r]\setminus I$, then $Q_i$ is a directed $x_j',x_j$-path of length $3(j-i^-(j))+1>1$, so that $i^-(j)<j$. As $x_j'\in V(R_{i^-(j)})\subseteq V(F_{i^-(j)})$, and $x_j\in V(S_j)\subseteq V(F_j)$,  \ref{prop7} holds for $Q_i$. Finally, if $Q_i=Q_j^-$ for some $j\in I$, then $S_j$ was chosen in \ref{case12}. From the choice of the relevant maximal directed path $S_{j}'$, the first vertex $x_j'$ of $Q_i$ is in $V(F_{j}^-)\setminus V(S_{j})$ and the last vertex $x_j$ of $Q_i$ is also the last vertex of $S_{j}$, and therefore \ref{prop8} holds.
\end{proof}

%%%%%%%%%%%%%%%%%%%%%%%%%%%%%%%%%%%%%%%%%%%%%%%%%%%%%%%%%%%%%%%%%%%%%%%%%%%%%%%%%%%%%%%%%%%%%%%%%%%%%%%%%
%%%%%%%%%%%%%%%%%%%%%%%%%%%%%%%%%%%%%%%%%%%%%%%%%%%%%%%%%%%%%%%%%%%%%%%%%%%%%%%%%%%%%%%%%%%%%%%%%%%%%%%%

\subsection{Embedding a good decomposition}\label{subsec:embedding_a_good_decomposition}

We now show that it is possible to embed an $(r,m)$-good decomposition $D$ of a $n$-vertex tree $T$ with $k$ leaves into an $(n+k-2)$-vertex tournament $G$, when $1/n\ll 1/r,1/k,m/n$. For our sketch we will use the notation of Definition~\ref{def:gooddecomp}. We take a median order of $G$ and find within it consecutive disjoint intervals $V_1,U_1,V_2,U_2,\ldots,V_r,U_r,V_{r+1}$ with carefully chosen sizes. We will embed $S_1$ into $G[V_1]$ while embedding its root to the last vertex of $V_1$ under $\sigma$, using Theorem~\ref{thm:DH}, and similarly embed $S_{r+1}$ into $V_{r+1}$ so that its root is embedded to the first vertex of $V_{r+1}$ under $\sigma$. For each $i\in\{2,\ldots,r\}$, we will have $|V_i|=|S_i|$ and embed the directed path $S_i$ into $G[V_i]$ using the ordering provided by $\sigma$.

As described at the start of this section, for each $i\in [r]$, we then find a short path $R_i$ from the last vertex of $V_i$ under $\sigma$ to the first vertex of $V_{i+1}$ under $\sigma$ which can `absorb' any subset of vertices from $U_i$ (see Claim~\ref{clm:Rismall}). We then embed the forests $F_i^+$, $F_i^-$, $i\in [r+1]$ and directed paths $Q_i$, $i\in [\ell]$, into $\cup_{i\in [r]}(U_i\setminus V(R_i))$, before incorporating the right number of vertices into each path $R_i$. More specifically, as depicted in Figure~\ref{fig:embedding}, for each $i\in [r]$, we will divide $U_i$ into six parts, $U_{i,1},\ldots,U_{i,6}$, again with carefully chosen sizes. The sets $U_{i,1}$ and $U_{i,6}$ will be small and covered by $R_i$ (aiding the desired `absorption' property of $R_i$). 
We will embed $V(F_i^+)\setminus V(S_i)$ into $U_{i,2}\setminus V(R_i)$, using \ref{prop4} and that typical vertices in $V_i$ (the image of $S_i$) have plenty of out-neighbours in $U_{i,2}$ (see Claim~\ref{clm:manynbr}) and $V(R_i)$ is small. Similarly, we will embed $V(F_{i+1}^-)\setminus V(S_{i+1})$ into $U_{i,4}\setminus V(R_i)$ (see also Claim~\ref{clm:manynbr}). We will embed paths $Q_j$ satisfying \ref{prop8} using the appropriate set $U_{i,5}$ (see Claim~\ref{clm:connectUi3short}).
We will then embed paths $Q_j$ satisfying \ref{prop7} using different sets $U_{i,3}$ (see Claim~\ref{clm:connectUi3}). 
As we chose the size of the sets $U_i$, $i\in [r]$, carefully, for each $i\in [r]$, we will then have the correct number of vertices unused in $U_i$ to absorb into $R_i$ and complete the embedding of $P_i$, and hence also the embedding of $T\subset D$.

\begin{lemma}\label{lem:embedgood} Let $1/n\ll \mu, 1/r,1/k$ and $m\geq \mu n$. Every tournament with $n+k-2$ vertices contains a copy of every $n$-vertex oriented tree with $k$ leaves which has an $(r,m)$-good decomposition.
\end{lemma}
\begin{proof} Note that we can additionally assume that $\mu\ll 1/r,1/k$. Let $G$ be an $(n+k-2)$-vertex tournament and suppose that the $n$-vertex tree $T$ with $k$ leaves has an $(r,m)$-good decomposition $D$ using the notation in Definition~\ref{def:gooddecomp}.
Let $k_0$ be the number of in-leaves of $S_1$ and let $k_1$ be the number of out-leaves of $S_{r+1}$. By \ref{prop5}, we have $k_0,k_1\geq 1$ and by \ref{prop6} we have $k_0+k_1\leq k$.

Let $I_1\subset [\ell]$ be the set of $i\in [\ell]$ satisfying \ref{prop7}. Let $I_2=[\ell]\setminus I_1$, so that, by \ref{prop78}, each $i\in I_2$ satisfies \ref{prop8}.
For each $i\in I_1$, using \ref{prop7}, let $q_i,r_i\in[r+1]$ with $q_i<r_i$ be such that $Q_i$ is a directed path from $F_{q_i}$ to $F_{r_i}$ with length $3(r_i-q_i)+1$. For each $i\in [r]$, let $a_i$ be the number of $j\in I_1$ for which $q_j\leq i<r_j$. For each $i\in I_2$, using \ref{prop8}, let $2\leq s_i\leq r$ be such that $Q_i$ is a directed path from $V(F_{s_i}^-)\setminus V(S_{s_i})$ to the last vertex of $S_{s_i}$. For each $i\in [r]$, let $b_i$ be the number of $j\in I_2$ with $s_j=i+1$ (and note that we always have $b_r=0$).

Let $\sigma=v_1,\ldots,v_{n+k-2}$ be a median order of $G$. Take in $v_1,\ldots, v_{n+k-2}$ consecutive disjoint intervals
\[
V_1,U_1,V_2,U_2,V_3,\ldots,V_r,U_{r},V_{r+1}
\]
such that $|V_1|=|S_1|+k_0-1$, $|V_{r+1}|=|S_{r+1}|+k_1-1$, and, for each $2\leq i\leq r$, $|V_i|=|S_i|$, and, for each $i\in [r]$,
\begin{equation}\label{Uisize}
|U_i|=|P_i|-2+|V(F^+_{i})\setminus V(S_i)|+|V(F^-_{i+1})\setminus V(S_{i+1})|+3a_i+2b_i\geq |P_i|-2\overset{\text{\ref{prop3}}}{\geq} 2000m-1 .
\end{equation}
Note that this is possible, as, by \ref{prop4} and \ref{prop5}, $|F^-_1|=|S_1|$ and $|F^+_{r+1}|=|S_{r+1}|$, so that, using \ref{prop4}, we have
\begin{align*}
\sum_{i=1}^{r+1}|V_i|&+\sum_{i=1}^{r}|U_i|=k_0+k_1-2+\sum_{i=1}^{r+1}|S_i|+\sum_{i=1}^{r}(|P_i|-2+|F^+_{i}|+|F^-_{i+1}|-|S_{i}|-|S_{i+1}|+3a_i+2b_i)\\
&\overset{\text{\ref{prop6}}}{\leq} k-2+|S_1|+|S_{r+1}|+\sum_{i=2}^{r}|S_i|+\sum_{i=1}^{r}(|P_i|-2)+\sum_{i=1}^{r+1}(|F^+_{i}|+|F^-_{i}|-2|S_{i}|)+\sum_{i\in [r]}(3a_i+2b_i)\\
&\overset{\eqref{eqn:Ppart}}{=}k-2+|S_1|+|S_{r+1}|+|P|-2+\sum_{i=1}^{r+1}(|F^+_{i}|+|F^-_{i}|-2|S_{i}|)+ 3\sum_{i\in I_1}(r_i-q_i)+2|I_2|\\
% &\hspace{-0.475cm}\overset{\text{\ref{prop7},\ref{prop8}}}= \;
&=k-2+|P|+(|S_1|+|S_{r+1}|-2)+\sum_{i=1}^{r+1}|(V(F^+_{i})\cup V(F^-_i))\setminus V(S_i)|+ \sum_{i\in [\ell]}(|Q_i|-2)\\
&\overset{\text{\ref{prop2}}}=n+k-2.
\end{align*}

Next, for each $i\in [r]$, partition $U_i$ as intervals $U_{i,1},\ldots,U_{i,6}$ in that order such that
\begin{equation}\label{eqn:Uijsizes}
|U_{i,1}|=m,\;|U_{i,2}|=10m,\; |U_{i,4}|=110m,\;|U_{i,5}|=100m,\; |U_{i,6}|=m
\end{equation}
\begin{equation}\label{eqn:Ui3size}
\text{ and }|U_{i,3}|=|U_i|-222m\overset{\eqref{Uisize}}\geq 1700m.
\end{equation}
Note also, by \ref{prop4}, that, for each $i\in \{2,\ldots,r\}$,
\begin{equation}\label{eqn:Visize}
    |V_i|=|S_i|\leq |F_i|\leq m.
\end{equation}

For each $i\in [r]$, let $U'_i$ be a subset of $U_i$ where each vertex is included uniformly at random with probability $\mu/20$. By Lemma~\ref{lm:basic_properties_of_median_orders}~ii) $v_1v_2\ldots v_{n+k-2}$ forms a directed path in that order, so there is a directed path from the last vertex of $V_i$ under $\sigma$ to the first vertex of $V_{i+1}$ under $\sigma$, whose vertex set covers $U_{i,1}\cup U'_i\cup U_{i,6}$ and whose vertex order is a suborder of $\sigma$. Let $R_i$ be a shortest such path. We now prove that, with positive probability, the `absorption property' we need for $R_i$ holds, as well as a bound on $|R_i|$.

\begin{claim}\label{clm:Rismall} With positive probability, for each $i\in [r]$, $|V(R_i)\setminus (U_{i,1}\cup U_{i,6})|\leq m$, so that $|R_i|\leq 3m$, and, for any $U\subset U_i\cup V(R_i)$ with $V(R_i)\subset U$, there is a directed path with the same start vertex and end vertex as $R_i$ but with vertex set $U$.
\end{claim}
\renewcommand\qedsymbol{$\boxdot$}
\begin{proof}[Proof of Claim~\ref{clm:Rismall}] Let $p=\mu/20$ and $i\in [r]$. Note that, by Lemma~\ref{chernoff}, as $|U_i|\leq n$ and $1/n\ll \mu,1/r$, we have, with probability at least $1-1/3r$ that $|U'_i|\leq 2pn$. For each $v\in U_i\setminus (U_{i,1}\cup U_{i,6})$, let $\mathbf{E}_v$ be the following event.
\begin{itemize}
\item[$\mathbf{E}_v$:] There are $u\in N^-(v)\cap U_i'$ and $u'\in N^+(v)\cap U_i'$ with $u<_\sigma v<_\sigma u'$.
\end{itemize}
Now, by Lemma~\ref{lm:basic_properties_of_median_orders}~ii), for each $v\in U_i\setminus (U_{i,1}\cup U_{i,6})$, we have
\[
|\{u\in N^-(v)\cap U_i:u<_\sigma v\}|\geq \frac{|\{u\in U_i:u <_\sigma v\}|}{2}\geq \frac{|U_{i,1}|}{2}\overset{\eqref{eqn:Uijsizes}}{=} \frac{m}{2},
\]
and
\[
|\{u\in N^+(v)\cap U_i:u>_\sigma v\}|\geq \frac{|\{u\in U_i:u >_\sigma v\}|}{2}\geq \frac{|U_{i,6}|}{2}\overset{\eqref{eqn:Uijsizes}}{=} \frac{m}{2},
\]
so that $\P(\mathbf{E}_v\text{ does not hold})\leq 2(1-p)^{m/2}\leq 2\exp(-pm/2)\leq 2\exp(-\mu^2 n/40)$. Therefore, as $1/n\ll \mu,1/r$, a union bound implies that, with probability at least $1-1/3r$, $\mathbf{E}_v$ holds for each $v\in U_i\setminus (U_{i,1}\cup U_{i,6})$. Thus, with probability at least 1/3, we have, for each $i\in [r]$, that $\mathbf{E}_v$ holds for each $v\in U_i\setminus (U_{i,1}\cup U_{i,6})$, and $|U'_i|\leq 2pn$. Assuming these events occur, we now prove that the property in the claim holds for each $i\in [r]$.

By Corollary~\ref{cor:last} and the minimality of $R_i$, any two vertices in $U_{i,1}\cup U_i'\cup U_{i,6}$ on $R_i$, with no vertices between them on $R_i$ from $U_{i,1}\cup U_i'\cup U_{i,6}$ 
have at most 1 vertex between them on $R_i$. As the vertices from $U_{i,1}\cup U_{i,6}$ form two intervals on $R_i$, just after the first vertex and just before the last vertex of $R_i$ respectively, $|V(R_i)\setminus (U_{i,1}\cup U_{i,6})|\leq 2+2|U_i'|+1\leq 4pn+3\leq m$.

Now, take any set $U\subset U_i\cup V(R_i)$ with $V(R_i)\subset U$. Let $R_U$ be a directed path with the same endvertices as $R_i$ which contains every vertex of $R_i$ in order according to $\sigma$ and for which $V(R_U)\subset U$, and which, under these conditions, has the maximum possible length. Note that this exists as $R_i$ itself satisfies these conditions. Suppose, for contradiction, that there is some $v\in U\setminus V(R_U)$. Note that $v\in U_i\setminus (U_{i,1}\cup U_{i,6})$. Let $\ell$ be the length of $R_U$ and label vertices so that $R_U=u_0u_1\ldots u_\ell$. As $\mathbf{E}_v$ holds and $U_i'\subset V(R_i)\subset V(R_U)$, we can take $j=\min\{j'\in \{0,1,\ldots,\ell\}:u_{j'}\in N^-(v)\}$ and find that $u_j<_\sigma v$. Let $j''\in \{0,1,\ldots,\ell\}$ be the smallest $j''>j$ such that $u_{j''}\in N^+(v)$.

Observe that, $u_{j''-1}\notin N^+(v)$, so that, as $G$  is a tournament, $u_{j''-1}\in N^-(v)$ and therefore
\[
u_0u_1\ldots u_{j''-1}vu_{j''}\ldots u_\ell,
\]
is a directed path with the same endvertices as $R_U$ (and hence $R_i$) which contains every vertex of $R_i$ in order according to $\sigma$. As this path has vertex set $\{v\}\cup V(R_U)\subset U$ and $v\notin V(R_U)$, this path contradicts the maximality of $R_U$. Thus, $V(R_U)=U$, so that $R_U$ is a directed path with the same endvertices as $R_i$ and with vertex set $U$, as required.
\end{proof}\renewcommand\qedsymbol{$\square$}

Assume then, that the property in Claim~\ref{clm:Rismall} holds. We now show three further claims, before embedding $T$. This embedding, annotated with which part of the embedding is done with each claim, is depicted in Figure~\ref{fig:embedding}. For each $i\in [r+1]$, we will use the following claim to embed the vertices in $V(F^+_i)\setminus V(S_i)$ to $U_{i,2}$ (if $i\neq r+1$) and embed the vertices in $V(F^-_{i})\setminus V(S_{i})$ to $U_{i-1,4}$ (if $i\neq 1$) so that they attach appropriately to an embedding of $S_i$ into the vertex set $V_i$.
%We will use the following claim to, for each $i\in [r]$, embed the vertices in $V(F^+_i)\setminus V(S_i)$ to $U_{i,2}$ and the vertices in $V(F^-_{i+1})\setminus V(S_{i+1})$ to $U_{i,4}$ so that they attach appropriately to an embedding of $S_i$ into the vertex set $V_i$.

\begin{figure}
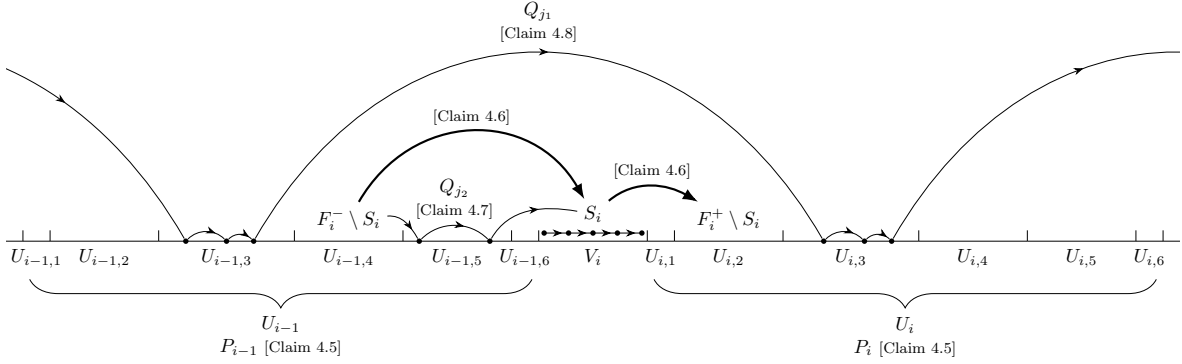

\DiagramEmbedding
\vspace{-0.75cm}
\caption{Embedding an $(r,m)$-good decomposition (as depicted in Figure~\ref{fig:gooddecomp}) into a median order, with the claims used to embed each part.}\label{fig:embedding}
\end{figure}

\begin{claim}\label{clm:manynbr}
For each $i\in [r]$ and $v\in V_i$, we have $|N^+(v,U_{i,2})\setminus V(R_i)|\geq 3m$, and, for each $i\in [r]$ and $v\in V_{i+1}$, we have $|N^-(v,U_{i,4})\setminus V(R_i)|\geq 3m$.
\end{claim}
\renewcommand\qedsymbol{$\boxdot$}
\begin{proof}[Proof of Claim~\ref{clm:manynbr}] Let $i\in [r]$ and $v\in V_i$, and take $V_{i,v}=\{u\in V_i:u>_\sigma v\}$. By Lemma~\ref{lm:basic_properties_of_median_orders}~ii), we have that
\begin{align*}
|N^+(v,U_{i,2})|&\geq |N^+(v,V_{i,v}\cup U_{i,1}\cup U_{i,2})|-|V_{i,v}\cup U_{i,1}|\geq \frac{|V_{i,v}\cup U_{i,1}\cup U_{i,2}|}{2}-|V_{i,v}\cup U_{i,1}|\\
&=\frac{|U_{i,2}|-|V_{i,v}\cup U_{i,1}|}{2}\geq\frac{|U_{i,2}|-|V_{i}\cup U_{i,1}|}{2}\overset{\eqref{eqn:Uijsizes},\eqref{eqn:Visize}}{\geq}  \frac{10m-m-m}{2}=4m.
\end{align*}
Therefore, by the property from Claim~\ref{clm:Rismall}, $|N^+(v,U_{i,2})\setminus V(R_i)|\geq |N^+(v,U_{i,2})|-(|R_i|-|U_{i,1}|-|U_{i,6}|)\geq 3m$.

Let then $i\in [r]$ and $v\in V_{i+1}$ and let $V'_{i+1,v}=\{u\in V_{i+1}:u<_\sigma v\}$.  By Lemma~\ref{lm:basic_properties_of_median_orders}~ii), we have similarly that
\begin{align*}
|N^-(v,U_{i,4})|&\geq |N^-(v,V'_{i+1,v}\cup U_{i,4}\cup U_{i,5}\cup U_{i,6})|-|V'_{i+1,v}\cup U_{i,5}\cup U_{i,6}|\\
&\geq \frac{|U_{i,4}|-|V'_{i+1,v}\cup U_{i,5}\cup U_{i,6}|}{2}\overset{\eqref{eqn:Uijsizes},\eqref{eqn:Visize}}{\geq}\frac{110m-100m-m-m}{2}=4m.
\end{align*}
Therefore,  by the property from Claim~\ref{clm:Rismall} again, $|N^-(v,U_{i,4})\setminus V(R_i)|\geq |N^-(v,U_{i,4})|-(|R_i|-|U_{i,1}|-|U_{i,6}|)\geq 3m$.
\end{proof}\renewcommand\qedsymbol{$\square$}

We will use the following claim, for each $i\in I_2$, to embed the path $Q_i$ when its first and last vertex have already been embedded into $U_{s_i-1,4}$ and $V_{s_i}$ respectively.

\begin{claim}\label{clm:connectUi3short}
For each $2\leq j\leq r$, $v\in U_{j-1,4}$, $w\in V_{j}$ and $U\subset U_{j-1,4}\cup U_{j-1,5}$ with $|U|\leq2m$, there is a directed $v,w$-path in $G$ with length 3 and internal vertices in $(U_{j-1,4} \cup U_{j-1,5})\setminus (U\cup V(R_{j-1}))$.
\end{claim}
\renewcommand\qedsymbol{$\boxdot$}
\begin{proof}[Proof of Claim~\ref{clm:connectUi3short}]
%Let $A_{j,v,w,U}=(V(R_{j-1})\cap U_{j-1,5})\cup U_{j-1,6}\cup \{u\in U\cup V_{j}:v<_\sigma u <_\sigma w\}$
Let $A_{j,v,w,U}=\{u\in U\cup V(R_{j-1})\cup V_{j}:v<_\sigma u <_\sigma w\}$ , and note that, by~\eqref{eqn:Visize} and the choice of $R_i$ according to Claim~\ref{clm:Rismall},  $|A_{j,v,w,U}|\leq 6m$. The number of vertices between $v$ and $w$ in $\sigma$ is at least $|U_{j-1,5}|+|U_{j-1,6}|=101m> 6|A_{j,v,w,U}|+8$. Therefore, by Lemma~\ref{lm:connecting_path_length_3}, there is a directed $v,w$-path in $G$ with length 3 and internal vertices in $\{u\notin A_{j,v,w,U}:v<_\sigma u <_\sigma w\}$. Because $U_{j-1,6}\subseteq V(R_{j-1})$, we have $\{u\notin A_{j,v,w,U}:v<_\sigma u <_\sigma w\}\subset (U_{j-1,4} \cup U_{j-1,5})\setminus (U\cup V(R_{j-1}))$, and so the claim holds.
\end{proof}\renewcommand\qedsymbol{$\square$}

For each $i\in [6]$, let $U_{0,i}=U_{r+1,i}=\emptyset$, and note that, by \ref{prop4} and \ref{prop6}, $|V_1|,|V_{r+1}|\leq m+k\leq 2m$. For each $i\in [r+1]$, let $\bar{V}_i=U_{i-1,4}\cup U_{i-1,5}\cup U_{i-1,6}\cup V_i\cup U_{i,1}\cup U_{i,2}$, and note that, by \eqref{eqn:Uijsizes} and \eqref{eqn:Visize}, $|\bar{V}_i|\leq 225m$. Note that $\bar{V}_1U_{1,3}\bar{V}_2U_{2,3}\ldots \bar{V}_rU_{r,3}\bar{V}_{r+1}$ are consecutive intervals in $\sigma$.

We will use the following claim, for each $i\in I_1$, to embed the path $Q_i$ when its first and last vertex have already been embedded into $\bar{V}_{q_i}$ and $\bar{V}_{r_i}$ respectively.

\begin{claim}\label{clm:connectUi3}
For each $1\leq i<j\leq r+1$, $v\in \bar{V}_i$, $w\in \bar{V}_j$ and $U\subset V(G)$ with $|U|\leq m$, there is a directed $v,w$-path in $G$ with length $3(j-i)+1$ and exactly 3 vertices in each set $U_{i',3}\setminus (U\cup V(R_{i'}))$, $i\leq i'<j$.
\end{claim}
\renewcommand\qedsymbol{$\boxdot$}
\begin{proof}[Proof of Claim~\ref{clm:connectUi3}]
First we will choose vertices $u_{i'}$, $i\leq i'<j$ between $u_{i-1}:=v$ and $w$ in the median order, with $u_{j-1}w\in E(G)$ before carefully applying Lemma~\ref{lm:connecting_path_length_3} to each consecutive pair of vertices in $v,u_i,u_{i+1},\ldots,u_{j-1}$ to get, together with $u_{j-1}w$, a $v,w$-path with length $3(j-i)+1$.

To do this, first, for each $i'$, $i\leq i'\leq j-2$, let $u_{i'}$ be the last vertex in $U_{i',3}\setminus (U\cup V(R_{i'}))$ under $\sigma$. Let $U'_{j-1,3}$ be the set of the last $250m$ vertices of $U_{j-1,3}$ under $\sigma$, and let $\bar{V}_{j,w}=\{w'\in \bar{V}_j:w'<_\sigma w\}$, so that $|\bar{V}_{j,w}|\leq |\bar{V}_j|\leq 225m$. Note that, by  Lemma~\ref{lm:basic_properties_of_median_orders}~ii), we have
\begin{align*}
|N^-(w,U'_{j-1,3})\setminus (U\cup V(R_{j-1}))|&\geq |N^-(w,\bar{V}_{j,w}\cup U'_{j-1,3})|-|\bar{V}_{j,w}|-|U\cup V(R_{j-1})|\\
&\geq \frac{|U'_{j-1,3}|-|\bar{V}_{j,w}|}{2}-|U\cup V(R_{j-1})|\geq\frac{250m-225m}{2}-4m>0.
\end{align*}
Let $u_{j-1}$ then be the last vertex of $N^-(w,U_{j-1,3})\setminus (U\cup V(R_{j-1}))$ under $\sigma$, noting that there are fewer than $250m$ vertices in $U_{j-1,3}$ after $u_{j-1}$ under $\sigma$. Let $u_{i-1}=v$.

For each $i\leq i'<j$, we will show there exists a directed $u_{i'-1},u_{i'}$-path $T_{i'}$ with length 3 and internal vertices in $U_{i',3}\setminus (U\cup V(R_{i'}))$. Noting that $T_{i}T_{i+1}\ldots T_{j-1}w$ is a directed path with length $3(j-i)+1$ and exactly three vertices in each set $U_{i',3}\setminus (U\cup V(R_{i'}))$, $i\leq i'<j$, will then complete the proof of the claim.

Let then $i\leq i'<j$ and let $A_{i'}=\{u\in U_{i'-1,3}\cup \bar{V}_{i'}\cup ((U\cup V(R_{i'}))\cap U_{i',3}):u_{i'-1}<_\sigma u<_\sigma u_{i'}\}$. Note that, for each $i\leq i'<j$, by the choice of $u_{i'}$ there are at most $|U\cup V(R_{i'-1})|\leq 4m$ vertices after $u_{i'-1}$ in $U_{i'-1,3}$ under $\sigma$, so $|A_{i'}|\leq 4m+225m+|U\cup V(R_{i'})|\leq 233m$. In addition, recall that there are fewer than $250m$ vertices in $U_{j-1,3}$ after $u_{j-1}$ under $\sigma$. Therefore, by \eqref{eqn:Ui3size}, for each $i\leq i'<j$, there are at least $1700m-250m> 6|A_{i'}|+8$ vertices in $U_{i',3}$ before $u_{i'}$ under $\sigma$. So, by Lemma~\ref{lm:connecting_path_length_3}, there is a directed $u_{i'-1},u_{i'}$-path $T_{i'}$ with length 3 and internal vertices in $\{u\notin A_{i'}:u_{i'-1}<_\sigma u <_\sigma u_{i'}\}\subset U_{i',3}\setminus (U\cup V(R_{i'}))$, as required.
\end{proof}
\renewcommand\qedsymbol{$\square$}

We are now ready to embed the $(r,m)$-good decomposition $D$ into $G$, as follows. Begin with the empty embedding $\phi:\emptyset\to V(G)$. For each $2\leq i\leq r$, recalling that $|V_i|=|S_i|$, extend $\phi$ to embed the directed path $S_i$ onto the vertices in $V_i$ in the order given by $\sigma$. Note that the vertices of each interval $V_i$ form a directed path in this order by Lemma~\ref{lm:basic_properties_of_median_orders}~ii).

Let $x'$ be the last vertex of $V_1$ under $\sigma$, and let $y'$ be the first vertex of $V_{r+1}$ under $\sigma$. Recall that $P$, as defined in \eqref{eqn:Ppart}, is a directed $x,y$-path, $S_1$ is an in-arboresence with $k_0$ in-leaves and root $x$, and $S_{r+1}$ is an out-arboresence with $k_1$ out-leaves and root $y$. Therefore, as $|V_1|=|S_1|+k_0-1$ and $|V_{r+1}|=|S_{r+1}|+k_1-1$, by  Theorem~\ref{thm:DH} (applied twice, once with directional duality) we can extend $\phi$ to embed $S_1$ into $V_1$ such that $\phi(x)=x'$ and embed $S_{r+1}$ into $V_{r+1}$ such that $\phi(y)=y'$.

Now, for each $i\in [r+1]$ and $v\in V(S_i)$, let $F_v^-$ be the component of $F_i^-$ containing $v$ and let $F_v^+$ be the component of $F_i^+$ containing $v$. For each vertex $v\in V(S_i)$ in increasing order of $\phi(v)$ under $\sigma$, greedily and disjointly extend $\phi$ to embed $F_v^--v$ into $N^-(\phi(v),U_{i-1,4})\setminus V(R_{i-1})$ and $F_v^+-v$ into $N^+(\phi(v),U_{i,2})\setminus V(R_{i})$.
Note this is possible for each $v\in V(S_i)$ as, by \ref{prop5}, if $|E(F_v^-)|>0$, then $i\geq 2$ and thus, by Claim~\ref{clm:manynbr},
\begin{align*}
|N^-(\phi(v),U_{i-1,4})\setminus (V(R_{i-1})\cup (\cup_{u\in V(S_i):\phi(u)<_\sigma \phi(v)} \phi (F_u^-)))|&\geq 3m-(|F_i^-|-|V(F_v^-)\setminus\{v\}|)\\
&\overset{\text{\ref{prop4}}}{\geq} 3|V(F_v^-)\setminus\{v\}|,
\end{align*}
so that a copy of $F_v^--v$ in $N^-(\phi(v),U_{i-1,4})\setminus (V(R_{i-1})\cup (\cup_{u\in V(S_i):\phi(u)<_\sigma \phi(v)} \phi (F_u^-)))$ exists by Theorem~\ref{cor:3n-3}. Similarly, for each $v\in V(S_i)$, this is also possible for $F_v^+-v$.

For each $i\in [\ell]$, say that $Q_i$ is a directed path from $x_i$ to $y_i$. For each $i\in[\ell]$ in turn, extend $\phi$ to cover $V(Q_i)\setminus \{x_i,y_i\}$, by doing the following.
\begin{itemize}
    \item If $i\in I_1$, recall that $q_i,r_i$ are such that $Q_i$ is a directed path from $F_{q_i}$ to $F_{r_i}$ with length $3(r_i-q_i)+1$, where $q_i<r_i$, and note that $\phi(x_i)\in \phi(V(F_{q_i}))\subset \bar{V}_{q_i}$ and $\phi(y_i)\in \phi(V(F_{r_i}))\subset \bar{V}_{r_i}$. Embed $Q_i$ as a directed $\phi(x_i),\phi(y_i)$-path with length $3(r_i-q_i)+1$ and exactly three vertices in $U_{i',3}\setminus (V(R_{i'})\cup (\cup_{j\in [i-1]}\phi(V(Q_j))))$, for each $q_i\leq i'<r_i$. Note that this is possible, by Claim~\ref{clm:connectUi3}, as when we look for such a path we have $|\cup_{j\in [i-1]}\phi(V(Q_i))|\leq \ell\cdot (3r+2)\leq m$ as $\ell\leqslant2r$, $1/n\ll \mu\ll 1/r$ and $m\geq \mu n$.
    \item If $i\in I_2$, recall that $2\leq s_i\leq r$ is such that $Q_i$ is a directed path with length 3 from $V(F_{s_i}^-)\setminus V(S_{s_i})$ to the last vertex of $S_{s_i}$, and note that $\phi(x_i)\in \phi(V(F_{s_i}^-)\setminus V(S_{s_i}))\subset U_{s_i-1,4}$ and $\phi(y_i)\in \phi(V(S_{s_i}))\subset V_{s_i}$. Embed $Q_i$ as a directed path with length 3 from $\phi(x_i)$ to $\phi(y_i)$ with interior vertices in $(U_{s_i-1,4}\cup U_{s_i-1,5})\setminus (\phi(V(F_{s_i}^-))\cup (\cup_{j\in [i-1]}\phi(V(Q_j)))\cup V(R_{s_i-1}))$. Note that this possible, by Claim~\ref{clm:connectUi3short}, as when we look for such a path we have, by \ref{prop4}, $|\phi(V(F_{s_i}^-))|+|\cup_{j\in [i-1]}\phi(V(Q_j))|\leq m+\ell\cdot (3r+2)\leqslant2m$.
\end{itemize}

Finally, we extend $\phi$ to cover the internal vertices of $P_i$, for each $i\in [r]$. For each $i\in [r]$, let $U''_i=\bigl(V(R_i)\cup U_i\bigr)\setminus \phi\bigl(V(F^+_{i})\cup V(F^-_{i+1})\cup (\cup_{j\in [\ell]}V(Q_j))\bigr)$. Note that $V(R_i)\cup U_i$ contains exactly the vertices in $U_i$ and the endvertices of $R_i$. Therefore,
\begin{align*}
|U''_i|&=|U_i|+2-(|F^+_{i}|-|S_{i}|)-(|F^-_{i+1}|-|S_{i+1}|)-3|\{j\in I_1:q_j\leq i<r_j\}|-2|\{j\in I_2:s_j=i+1\}|\\
&\overset{\eqref{Uisize}}=(|P_i|+3a_i+2b_i)-3a_i-2b_i=|P_i|.
\end{align*}
By Claim~\ref{clm:Rismall}, for each $i\in [r]$, there is a directed path between the endvertices of $R_i$ with vertex set $U''_i$. Using these paths, for each $i\in [r]$, extend the embedding $\phi$ to cover $P_i$, for each $i\in [r]$. This completes the embedding $\phi$ of $D=P\cup S_1\cup S_{r+1}\cup (\cup_{i\in [r+1]} F_i)\cup (\cup_{i\in [\ell]}Q_i)$, and hence, by \ref{prop1}, $G$ contains a copy of $T$.
\end{proof}

\subsection{Proof of Theorem~\ref{thm:fewcst}}\label{subsec:final}
Given Lemmas~\ref{lem:gooddecomp} and~\ref{lem:embedgood}, it is now straight-forward to prove Theorem~\ref{thm:fewcst}.

\begin{proof}[Proof of Theorem~\ref{thm:fewcst}] Note that, due to the result of Thomason~\cite{THO} quoted in the introduction, we may assume that $k\geq 3$. 
Let $n_0$ and $\mu$ be such that $1/n_0\ll \mu\ll 1/k$. Let $T$ be a tree with $n\geq n_0$ vertices and $k$ leaves, and let $G$ be a tournament with $n+k-2$ vertices.

If there are no vertices $x$ and $y$ with degree 2 in $T$ and a bare $x,y$-path $P$ with length at least $\mu n$ with  first and last block of length 1, then, by Lemma~\ref{lem:gooddecomp}, $T$ has an $(r,m)$-good decomposition for some $m\geq \mu n$ and $r\leq 10k$. In this case, then, by Lemma~\ref{lem:embedgood}, $G$ contains a copy of $T$. Thus, we can assume that $T$ contains vertices $x$ and $y$ with degree 2 in $T$ and a bare $x,y$-path $P$ with length at least $\mu n$ with  first and last block of length 1. 

Suppose first, that $k=3$. Note that in this case $P$ must lie in a maximal bare path of $T$ with one endvertex that is a leaf. Say this leaf is $z$, and assume, by relabelling $x$ and $y$ if necessary, that the path, $Q$ say, from $x$ to $z$ in $T$ contains $y$ (and hence $P$). Let $T'=T-(V(Q)\setminus\{x\})$.
Noting that $x$ is a leaf of $T'$, duplicate $x$ to get the tree $T''$ with the new leaf $x'$. Note that $T''$ has $4$ leaves and $|T|-|Q|+2\leq n-\mu n+1$ vertices. Therefore, by Theorem~\ref{thm:fewlinear}, as $1/n\ll\mu,1/k$, $G$ contains a copy of $T''$, $S''$ say. Let $s$ and $s'$ be the copy of $x$ and $x'$ in $S''$ respectively. Note that $|G-(V(S'')\setminus\{s,s'\})|=n+1-(n-|Q|)=|Q|+1$. By Theorem~\ref{thm:appending_non-directed_path}, there is a copy, $Q'$ say, of $Q$ with $x$ embedded to $\{s,s'\}$. Then $S''\cup Q'$ gives a copy of $T$.

Therefore, we have that $k\geq4$. In this case, let $T'=T-(V(P)\setminus\{x,y\})$. Noting that $x$ and $y$ are leaves of $T'$, create $T''$ by duplicating $x$ and $y$ to get the new vertices $x'$ and $y'$ respectively, and adding the edge $xy$. Note that $T''$ has $k+2$ leaves and $|T|-|P|+4\leq n-\mu n+3$ vertices.
 Therefore, by Theorem~\ref{thm:fewlinear}, as $1/n\ll \mu, 1/k$, $G$ contains a copy of $T''$, $S''$ say. Let $s,s',t$ and $t'$ be the copy of $x,x',y$ and $y'$ in $S''$ respectively. Note that $|G-(V(S'')\setminus\{s,s',t,t'\})|=n+k-2-(n-|P|)=|P|+k-2\geq |P|+2$.
By Theorem~\ref{thm:connecting_non-directed_paths}, there is a copy, $P'$ say, of $P$ with $x$ embedded to $\{s,s'\}$ and $y$ embedded to $\{t,t'\}$. Observing that $S''\cup P'$ contains a copy $T$ completes the proof that $G$ contains a copy of $T$ in this case.
\end{proof}

\bibliographystyle{abbrv}
\bibliography{references}

\begin{thebibliography}{10}

\bibitem{alon2004probabilistic}
N.~Alon and J.~H. Spencer.
\newblock {\em {\textbf{\emph{The {P}robabilistic {M}ethod}}}}.
\newblock John Wiley \& Sons, 2004.

\bibitem{CER-HAV}
S.~Ceroi and F.~Havet.
\newblock Trees with three leaves are $(n+1)$-unavoidable.
\newblock {\em Discrete Applied Mathematics}, 141:19--39, 2004.

\bibitem{DRO-HAV}
F.~Dross and F.~Havet.
\newblock On the unavoidability of oriented trees.
\newblock {\em Electronic Notes in Theoretical Computer Science}, 346:425--436,
  2019.

\bibitem{El_S}
A.~El~Sahili.
\newblock Trees in tournaments.
\newblock {\em Journal of Combinatorial Theory, Series B}, 92:183--187, 2004.

\bibitem{HAE-THO}
R.~H{\"a}ggkvist and A.~Thomason.
\newblock Trees in tournaments.
\newblock {\em Combinatorica}, 11:123--130, 1991.

\bibitem{HAN}
C.~B. Hanna.
\newblock Paths in tournaments a simple proof of {R}osenfeld's conjecture.
\newblock {\em arXiv preprint arXiv 2011.14394}, 2020.

\bibitem{HAV}
F.~Havet.
\newblock Trees in tournaments.
\newblock {\em Discrete Mathematics}, 243(1):121--134, 2002.

\bibitem{HAV2}
F.~Havet.
\newblock On unavoidability of trees with {$k$} leaves.
\newblock {\em Graphs and Combinatorics}, 19:101--110, 2003.

\bibitem{havet2000median}
F.~Havet and S.~Thomass{\'e}.
\newblock Median orders of tournaments: A tool for the second neighborhood
  problem and {S}umner's conjecture.
\newblock {\em Journal of Graph Theory}, 35(4):244--256, 2000.

\bibitem{HAV-THO}
F.~Havet and S.~Thomass\'e.
\newblock Oriented {H}amiltonian paths in tournaments: a proof of {R}osenfeld's
  conjecture.
\newblock {\em Journal of Combinatorial Theory, Series B}, 78(2):243--273,
  2000.

\bibitem{KUE-MYC-OST-2}
D.~Kühn, R.~Mycroft, and D.~Osthus.
\newblock A proof of {S}umner’s universal tournament conjecture for large
  tournaments.
\newblock {\em Proceedings of the London Mathematical Society},
  102(4):731–766, 2010.

\bibitem{reid1983embedding}
K.~Reid and N.~Wormald.
\newblock Embedding oriented {$n$}-trees in tournaments.
\newblock {\em Studia Sci. Math. Hungar}, 18(2-4):377--387, 1983.

\bibitem{rodl2006dirac}
V.~R{\"o}dl, A.~Ruci{\'n}ski, and E.~Szemer{\'e}di.
\newblock A {D}irac-type theorem for {$3$}-uniform hypergraphs.
\newblock {\em Combinatorics, Probability \& Computing}, 15(1-2):229, 2006.

\bibitem{ROS}
M.~{R}osenfeld.
\newblock Antidirected {H}amiltonian paths in tournaments.
\newblock {\em Journal of Combinatorial Theory, Series B}, 12(1):93--99, 1972.

\bibitem{THO}
A.~Thomason.
\newblock Paths and cycles in tournaments.
\newblock {\em Transactions of the American Mathematical Society},
  296:167--180, 1986.

\end{thebibliography}

\end{document}